\documentclass[final,1p,times,authoryear]{elsarticle}
\usepackage{amsmath}
\usepackage{amsthm}
\usepackage{amssymb}
\usepackage{amsfonts}
\usepackage{color,mathdots}
\usepackage{url}

\newtheorem{theorem}{Theorem}
\newtheorem{lemma}[theorem]{Lemma}
\newtheorem{corollary}[theorem]{Corollary}

\newtheorem{proposition}[theorem]{Proposition}
\newtheorem{definition}[theorem]{Definition}

\newtheorem{example}[theorem]{Example}
\newtheorem{remark}[theorem]{Remark}

\newtheorem{fact}[theorem]{Fact}

\RequirePackage{color}
\usepackage[pdfstartview=FitH,
            colorlinks,
           linkcolor=reference,
            citecolor=citation,
            urlcolor=e-mail]{hyperref}
\definecolor{e-mail}{rgb}{0,.40,.80}
\definecolor{reference}{rgb}{.20,.60,.22}
\definecolor{citation}{rgb}{0,.40,.80}

\usepackage{enumitem, hyperref}
\makeatletter
\def\namedlabel#1#2{\begingroup
    #2%
    \def\@currentlabel{#2}%
    \phantomsection\label{#1}\endgroup
}
\makeatother

\definecolor{answer}{rgb}{1,0,0}

\numberwithin{equation}{section}

\newcommand{\N}{{\mathbb N}}
\newcommand{\Q}{{\mathbb Q}}

\def \dd {\partial}
\def \D {\Delta}

\def \Nn {\N^m\times {\mathfrak n}}
\def \M {\mathcal M}

\def \L {\mathfrak L}
\def \n {\mathfrak n}
\def \LUB {\operatorname{LUB}}
\def \l {\langle}
\def \r {\rangle}
\def \ord {\operatorname{ord}}

\definecolor{todo}{rgb}{0,0,1}

\begin{document}

\begin{frontmatter}

\title{Effective bounds for the consistency of differential equations}

\author{Richard Gustavson\fnref{RG}}
\address{Manhattan College\\
Department of Mathematics\\
4513 Manhattan College Parkway\\
Riverdale, NY 10471}
\ead{rgustavson01@manhattan.edu}

\author{Omar Le\'on S\'anchez}
\address{University of Manchester\\
School of Mathematics\\
Oxford Road \\
Manchester, M13 9PL}
\ead{omar.sanchez@manchester.ac.uk}

\begin{abstract}
One method to determine whether or not a system of partial differential equations is consistent is to attempt to construct a solution using merely the ``algebraic data" associated to the system. In technical terms, this translates to the problem of determining the existence of regular realizations of differential kernels via their possible prolongations. In this paper we effectively compute an improved upper bound for the number of prolongations needed to guarantee the existence of such realizations, which ultimately produces solutions to many types of systems of partial differential equations. This bound has several applications, including an improved upper bound for the order of characteristic sets of prime differential ideals. We obtain our upper bound by proving a new result on the growth of the Hilbert-Samuel function, which may be of independent interest. 
\end{abstract}

\fntext[RG]{R. Gustavson was supported by NSF grants CCF-0952591, CCF-1563942, and DMS-1413859.}

\end{frontmatter}

\tableofcontents

\section{Introduction}
In this paper we study techniques that effectively determine if a given system of algebraic partial differential equations is consistent; that is, if the system has a solution in a differential field extension of the ground differential field in which the coefficients of the system live.  Our approach is to study the set of algebraic solutions of a given system of algebraic differential equations (viewed as a purely algebraic system), and then determine if an algebraic solution can be used to construct a differential solution.  This construction is not always possible, as evidenced by very basic examples such as the following:
\begin{equation}\label{motivating}
\begin{cases} \partial_1 u = u \\ \partial_2 u = 1 \end{cases}
\end{equation}
where $u$ is a differential indeterminate over some ground differential field with two commuting derivations $\partial_1$ and $\partial_2$. If we consider the associated algebraic system obtained by replacing $u$, $\partial_1u$, and $\partial_2u$ with algebraic indeterminates $x$, $z_1$, and $z_2$, respectively, we obtain
$$
\begin{cases} z_1 = x \\ z_2 = 1, \end{cases}
$$
which has a solution. 
However, the differential system (\ref{motivating}) is inconsistent, since the existence of a differential solution $a$ in some differential field would imply $1=\partial_2 \partial_1 a=\partial_1\partial_2 a=0 $. It is important to note that the inconsistency of the system becomes apparent after differentiating the system once. The number of differentiations needed to reveal that a given system is inconsistent is the main motivation of this paper.  Furthermore, we seek to effectively determine this number from data obtained from the equations (their order and the number of derivations and indeterminates).

To make the above discussion more precise, we study \emph{differential kernels}, which are field extensions of the ground differential field $(K,\partial_1,\dots,\partial_m)$ obtained by adjoining a solution of the associated algebraic system such that this solution serves as a means to ``prolong'' the derivations from $K$ (see Definition~\ref{defdiffker} for the precise definition of differential kernels).  Differential kernels in a single derivation were studied by Cohn \citep{Cohn} and Lando \citep{Lando}. In Section \ref{diffkerpre}, we consider differential kernels with an arbitrary number of commuting derivations.  A differential kernel is said to have a \emph{regular realization} if there is a differential field extension of $K$ containing the differential kernel and such that the generators of the kernel form the sequence of derivatives of the generators of order zero. The key observation is that a differential kernel has a regular realization if and only if the chosen solution of the 
associated algebraic system (i.e., the generators of the differential kernel) can be prolonged to yield a differential solution to the original system of differential equations. Thus, the problem of determining the consistency of a given system of differential equations is equivalent to the problem of determining the existence of regular realizations of a given differential kernel.  In a single derivation, every differential kernel has a regular realization \citep[Proposition 3]{Lando}. However, this is no longer the case with more than one derivation, as evidenced by the system \eqref{motivating} above, which is also discussed in Example \ref{noprolon} below.

The first analysis of differential kernels with several commuting derivations appears in the work of Pierce \citep{Pierce}, using different terminology (there a differential kernel is referred to as a field extension satisfying the \emph{differential condition}). In that paper it is shown that if a differential kernel has a \emph{prolongation} of a certain length (that is, we can extend the derivations from the algebraic solution some finite number of times), then it has a regular realization; see Theorem \ref{Pierce410} below.  We note here that even if a differential kernel has a proper prolongation, this is no guarantee that a regular realization will exist, as evidenced by Example \ref{examr} below.  We denote by $T_{r,m}^n$ the smallest prolongation length that guarantees the existence of a regular realization of any differential kernel of length $r$ in $n$ differential indeterminates over any differential field of characteristic zero with $m$ commuting derivations; see Definition \ref{thebignumber}. 
Note that this number only depends on the data $(r,m,n)$; in particular, it does \emph{not} depend on the degree of the algebraic system associated to the differential kernel.  A recursive construction of an upper bound for $T_{r,m}^n$ was provided in \citep[\S3]{LeOv}; unfortunately, this upper bound is unwieldy from a computational standpoint even when $m=2$ or $3$.

In this paper, we provide a new and improved upper bound for $T_{r,m}^n$. This new upper bound is given in Theorem \ref{construct} by the number $C_{r,m}^n$, which we introduce in Section \ref{expri}. The central idea for the construction of $C_{r,m}^n$ comes from weakening a condition imposed on what are called the minimal leaders of a differential kernel that guarantees the existence of a regular realization (compare conditions \eqref{condition1} and \eqref{condition2}). In further sections we show that there is a recursive algorithm that computes the value of the integer $C_{r,m}^n$. This is a nontrivial task, as we need to develop a series of new combinatorial results in order to complete the proof. In Section \ref{Macaulay}, we prove the main combinatorial result of the paper, Theorem \ref{strictmacaulay}. This theorem is a strengthening of Macaulay's theorem on the growth of the Hilbert-Samuel function when applied to \emph{connected antichain sequences} of $\N^m$ (see Definition \ref{connected}).  We then use a consequence of this combinatorial result, namely Corollary \ref{HSdeg}, in Section \ref{algorithm} to show that the integer $C_{r,m}^n$ can be expressed in terms of the maximal length of certain antichain sequences (see Theorem~\ref{UpperBound2}). At this point, 
we use the results from \citep[\S3]{LeOv} to derive an algorithm that computes the number $C_{r,m}^n$.  

This new upper bound $C_{r,m}^n$ of $T_{r,m}^n$ allows us to produce specific, computationally viable upper bounds for a small numbers of derivations (for example, one, two, or three derivations), which the previously known bound does not produce. At the end of Section \ref{diffkerpre} we provide some concrete computations to show how our new upper bound compares with what was previously known. For instance, our bound produces 
$$T_{r,2}^n\leq 2^nr \quad \text{ and }\quad T_{r,3}^1\leq 3(2^r-1),$$ 
which, surprisingly, was not known previously.

Having an effective bound for determining the existence of a regular realization of a differential kernel has several applications in computational differential algebra. In fact, these applications were our motivation to study differential kernels. We consider some of these in Section~\ref{applications}; namely:
\begin{enumerate}
 \item The bound $C_{r,m}^n$ produces an upper bound for the order of elements of a characteristic set (with respect to the canonical orderly ranking) of each minimal prime differential ideal containing a given collection of differential polynomials, answering a question first posed by Seidenberg in \citep{Seidenberg} and improving upon the bound given in \citep{Kondratieva}. An additional important feature of this new bound is that, in contrast with the one found in \citep{Kondratieva}, it does \emph{not} depend on the degrees of the given collection of differential polynomials; in fact, merely the existence of such a bound with no assumption on the degrees seems to be a new (and nontrivial) result. 
 \item The bound $C_{r,m}^n$ also produces an upper bound for the order of each irreducible component of finite order of a differential algebraic variety. This extends a well known result of Ritt \citep[Chapter 6]{Ritt} to the case of several commuting derivations. Again, we note that this bound does \emph{not} depend on the degrees of the defining differential polynomials.
 \item The number $T_{r,m}^n$ is used to determine an upper bound for the effective differential Nullstellensatz, which allows for the implementation of an algorithm that can check whether a given system of algebraic differential equations is consistent or not.  This problem was also first introduced by Seidenberg in \citep{Seidenberg}, with improvements in \citep{DJS, GKOS, Grigoriev}.  The current optimal upper bound is given in \citep{GKO} in terms of $T^n_{r,m}$, and our results show that for $m=2$ or $3$ this upper bound is computationally feasible.
 \item The number $T_{r,m}^n$ is also used in \citep{FreLe} to determine an upper bound for the degree of the Zariski closure of an affine differential algebraic variety. Our results show that for $m=2$ this bound is triple-exponential in $n$ and $r$.
\end{enumerate}

\medskip

\noindent {\bf Acknowledgments.} The authors would like to thank two anonymous referees for their detailed comments and suggestions, which led to improvements of a previous version of this manuscript.

\section{Differential kernels and preliminaries}\label{diffkerpre}

We work over a fixed differential field $(K,\D)$ of characteristic zero with $m$ commuting derivations $\Delta = \{\partial_1,\dots,\partial_m\}$.  Fix a postive integer $n$. We are interested in field extensions of $K$ whose generators over $K$ are indexed by elements of $\N^m\times {\n}$, where $\N=\{0,1,2,\dots\}$ and ${\n} = \{1,\dots,n\}$.  To do so, we introduce the following terminology: Given an element $\xi = (u_1,\dots,u_m) \in \N^m$, we define the {\it degree} of $\xi$ to be 
$$
\deg \xi = u_1 + \cdots + u_m.
$$
If $\alpha=(\xi,i)\in \Nn$, we set $\deg\alpha=\deg \xi$. For any $r \in \N$, we let 
$$
\Gamma(r) = \{\alpha \in \Nn : \deg \alpha \leq r\}.
$$

We will consider two different orders $\leq$ and $\unlhd$ on $\N^m \times {\n}$.  Given two elements $\alpha=(\xi,i)$ and $\beta=(\tau,j)$ of $\Nn$, we set $\alpha \leq \beta$ if and only if $i = j$ and $\xi \leq \tau$ in the product order of $\N^m$. On the other hand, if $\xi=(u_1,\dots,u_m)$ and $\tau=(v_1,\dots,v_m)$, we set $(\xi,i) \unlhd (\tau,j)$ if and only if 
$$
(\deg \xi,i,u_1,\dots,u_{m}) \;\text{ is less than or equal to }\; (\deg \tau,j,v_1,\dots,v_m)
$$ 
in the (left) lexicographic order.  Note that if $x=(x_1,\ldots,x_{n})$ are differential indeterminates and we identify $\alpha=(\xi,i)$ with $\dd^\xi x_i:=\dd_1^{u_1}\cdots\dd_m^{u_m}x_i$, then $\leq$ induces an order on the set of algebraic indeterminates $\{\dd^\xi x_i:(\xi,i)\in\Nn\}$ given by $\dd^\xi x_i\leq \dd^\tau x_j$ if and only if $\dd^\tau x_j$ is a derivative of $\dd^\xi x_i$ (in particular this implies that $i=j$). On the other hand, the ordering $\unlhd$ induces the canonical orderly ranking on the set of algebraic indeterminates.

Recall that an antichain of $(\Nn,\leq)$ is a subset of $\Nn$ of incomparable elements with respect to $\leq$. By Dickson's lemma every antichain must be finite. An {\it antichain sequence} of $(\N^m \times {\n},\leq)$ is a (finite) sequence $\bar \alpha=(\alpha_1,\dots,\alpha_k)$ of $\N^m \times {\n}$ such that $\alpha_i$ and $\alpha_j$ are incomparable when $i\neq j$.

We will look at field extensions of $K$ of the form
\begin{equation}\label{extL} 
L:=K(a^\xi_i : (\xi,i) \in \Gamma(r))
\end{equation}
for some fixed $r\in \mathbb N$, although occasionally we will have to consider extensions of the form $K(a^\xi_i: (\xi,i) \lhd (\tau,k))$ for some fixed $(\tau,k) \in \N^m \times {\n}$. Here we use $a^\xi_i$ as a way to index the generators of $L$ over $K$. The element $(\tau,j)\in \Nn$ is said to be a {\it leader} of $L$ if there is $\eta\in \N^m$ with $\eta\leq \tau$ and $\deg \eta\leq r$ such that $a^\eta_j$ is algebraic over $K(a^\xi_i : (\xi,i) \lhd (\eta,j))$, and a leader $(\tau,j)$ is a {\it minimal leader} of $L$ if there is no leader $(\xi,i)$ with $(\xi,i) < (\tau,j)$.  The set of minimal leaders of $L$ forms an antichain of $(\N^m \times {\n},\leq)$. We note that the notions of leader and minimal leader make sense even when we allow $r=\infty$.

\begin{definition}\label{defdiffker}
The field extension $L$, as in (\ref{extL}), is said to be a {\it differential kernel} over $K$ if there exist derivations 
$$
D_k:K(a^\xi_i : (\xi,i) \in \Gamma(r-1)) \to L
$$
extending $\partial_k$ for $1 \leq k \leq m$ such that $D_ka^\xi_i = a^{\xi + {\bf k}}_i$ for all $(\xi,i) \in \Gamma(r-1)$, where ${\bf k} \in \N^m$ is the $m$-tuple with a one in the $k$-th component and zeros elsewhere. The number $r$ is called the \emph{length} of the differential kernel. If $L$ has the form $K(a^\xi_i : (\xi,i) \lhd (\tau,j))$ for some fixed $(\tau,j) \in \N^m \times {\n}$, we say that $L$ is a differential kernel over $K$ if there exist derivations 
$$D_k:K(a^\xi_i : (\xi + {\bf k},i) \lhd (\tau,j)) \to L$$ 
extending $\partial_k$ for $1 \leq k \leq m$ such that $D_ka^\xi_i = a^{\xi + {\bf k}}_i$ whenever $(\xi + {\bf k},i) \lhd (\tau,j)$.
\end{definition}

Unless stated otherwise every differential kernel $L$ will have the form (\ref{extL}). 

\begin{definition}
A \emph{prolongation} of a differential kernel $(L,D_1,\dots,D_m)$ of \emph{length} $s\geq r$ is a differential kernel $L' = K(a^\xi_i : (\xi,i) \in \Gamma(s))$ over $K$ with derivations $D_1',\dots,D_m'$ such that $L'$ is a field extension of $L$ and  $D_k'$ extends $D_k$ for $1 \leq k \leq m$. The prolongation $L'$ of $L$ is called \emph{generic} if the set of minimal leaders of $L$ and $L'$ coincide.
\end{definition} 

In the ordinary case, $m=1$, every differential kernel of length $r$ has a prolongation of length $r+1$ (in fact a generic one) \citep[Proposition 1]{Lando}. However, for $m>1$, prolongations need not exist. 

\begin{example}\label{noprolon}
Working with $m=2$ and $n=1$, set $K=\Q$ and $L=\Q(t,t,1)$ where $t$ is transcendental over $\Q$. Here we are setting 
$$a^{(0,0)}=t, \quad a^{(1,0)}=t, \quad \text{ and } \quad a^{(0,1)}=1.$$ 
The field $L$ equipped with derivations $D_1$ and $D_2$ such that $D_1(t)=t$ and $D_2(t)=1$ is a differential kernel over $\Q$ of length $1$; however, it does not have a prolongation of length $2$. Indeed, if $L$ had a prolongation 
$$
L'=\Q(a^\xi:\deg\xi\leq 2)
$$
with derivations $D'_1$ and $D_2'$, then we would get the contradiction
$$
0=D_1' (1)=D_1'a^{(0,1)}=a^{(1,1)}=D_2'a^{(1,0)}=D_2'(t)=1.
$$
\end{example}

\begin{definition}
A differential kernel $L'=K(b_i^\xi:(\xi,i)\in\Gamma(r))$ is said to be a specialization (over $K$) of the differential kernel $L$ if the tuple $(b_i^\xi:(\xi,i)\in \Gamma(r))$ is a specialization of $(a_i^\xi:(\xi,i)\in \Gamma(r))$ over $K$ in the algebraic sense, that is, there is a $K$-algebra homomorphism 
$$
\phi:K(a_i^\xi:(\xi,i)\in \Gamma(r))\to K(b_i^\xi:(\xi,i)\in\Gamma(r))
$$
that maps $a_i^\xi\mapsto b_i^\xi$. The specialization is said to be \emph{generic} if $\phi$ is an isomorphism.
\end{definition}

\begin{lemma}\label{prolon}
Suppose $L'$ is a generic prolongation of $L$ of length $s$. If $\bar L$ is another prolongation of $L$ of length $s$, then $\bar L$ is a specialization of $L'$.
\end{lemma}
\begin{proof}
Let $L'=K(a_i^\xi:(\xi,i)\in\Gamma(s))$ with derivations $D'_1,\dots,D'_m$, and $\bar L=K(b_i^\xi:(\xi,i)\in \Gamma(s))$ with derivations $\bar D_1,\dots,\bar D_m$. Since $\bar L$ is a prolongation of $L$, we have that $b_i^\xi=a_i^\xi$ for all $(\xi,i)\in \Gamma(r)$. For convenience of notation we let 
$$
L'_{\unlhd(\tau,j)} := K(a_i^\xi:(\xi,i)\unlhd (\tau,j)) \; \text{ and } \; L'_{\lhd(\tau,j)} := K(a_i^\xi:(\xi,i)\lhd (\tau,j)),
$$
when $r\leq \deg \tau\leq s$. Note that 
$$
L'_{\unlhd(r{\bf 1},n)}=L \; \text{ and } \; L'_{\unlhd(s{\bf 1},n)}=L'.
$$
Similar notation, and remarks, apply to $\bar L_{\unlhd(\tau,j)}$ and $\bar L_{\lhd(\tau,j)}$. 

We prove the lemma by constructing the desired $K$-algebra homomorphism 
$$
\phi:L'_{\unlhd(\tau,j)}\to \bar L_{\unlhd(\tau,j)}
$$
recursively where $ (r{\bf 1},n)\unlhd (\tau,j)\unlhd (s{\bf 1},n)$. The base case, $(\tau,j)=(r{\bf 1},n)$, is  trivial since then 
$$L'_{\unlhd(r{\bf 1},n)}=L=\bar L_{\unlhd(r{\bf 1},n)}.$$ 
Now assume $(r{\bf 1},n)\lhd (\tau,j)\unlhd (s{\bf 1},n)$ and that we have a $K$-algebra homomorphism $\phi':L'_{\lhd(\tau,j)}\to \bar L_{\lhd(\tau,j)}$ mapping $a_i^\xi\mapsto b_i^\xi$ for $(\xi,i)\lhd(\tau,j)$. If $(\tau,j)$ is not a leader of $L'$, then $a_j^\tau$ is transcendental over $L'_{\lhd(\tau,j)}$, and so $\phi'$ extends to the desired $K$-algebra homomorphism $\phi:L'_{\unlhd(\tau,j)}\to\bar L_{\unlhd(\tau,j)}$. 

Hence, it remains to show the case when $(\tau,j)$ is a leader of $L'$. In this case, since $L'$ is a generic prolongation of $L$, $(\tau,j)$ is a nonminimal leader of $L'$, and moreover $a_j^\tau=(D')^\zeta a_j^\eta$ for some minimal leader $(\eta,j)$ of $L$ and nonzero $\zeta\in \N^m$. Let $f$ be the minimal polynomial of $a_j^\eta\in L$ over $K(a_i^\xi:(\xi,i)\lhd (\eta,j))$. The standard argument (in characteristic zero) to compute the derivative of an algebraic element in terms of its minimal polynomial yields a polynomial $g$ over $L'_{\lhd(\tau,j)}$ and a positive integer $\ell$ such that
$$
a_j^\tau=(D')^\zeta a_j^\eta=\frac{g(a_j^\eta)}{(f'(a_j^\eta))^\ell}\in L'_{\lhd(\tau,j)}.
$$
Similarly, there is a polynomial $h$ over $\bar L_{\lhd(\tau,j)}$ such that
$$
b_j^\tau=\bar D^\zeta a_j^\eta=\frac{h(a_j^\eta)}{(f'(a_j^\eta))^\ell}\in \bar L_{\lhd(\tau,j)},
$$
and, moreover, one such $h$ is obtained by applying $\phi'$ to the coefficients of $g$. This latter observation, together with the two equalities above, imply that $L'_{\unlhd(\tau,j)}=L'_{\lhd(\tau,j)}$ and that $\phi'(a_j^\tau)=b_j^\tau$. Hence, in the case when $a_j^\tau$ is a leader, setting $\phi:=\phi'$ yields the desired $K$-algebra homomorphism.
\end{proof}

\begin{definition}
An $n$-tuple $g=(g_1,\dots,g_n)$ contained in a differential field extension $(M,\dd_1',\dots,\dd_m')$ of $(K,\D)$ is said to be a \emph{regular realization} of the differential kernel $L$ if the tuple 
$$
((\dd')^\xi g_i:(\xi,i)\in\Gamma(r))
$$
is a generic specialization of $(a_i^\xi:(\xi,i)\in \Gamma(r))$ over $K$. The tuple $g$ is said to be a \emph{principal realization} of $L$ if there exists an infinite sequence of differential kernels $L=L_0,L_1,\dots$ of strictly increasing lengths, each a generic prolongation of the preceding, such that $g$ is a regular realization of each $L_i$. 
\end{definition}

Note that the differential kernel $L$ has a regular realization if and only if there exists a differential field extension $(M,\dd_1',\dots,\dd_m')$ of $(K,\D)$ such that $L$ is a subfield of $M$ and $\dd'_ka^\xi_i = a^{\xi + {\bf k}}_i$ for all $(\xi,i) \in \Gamma(r-1)$ and $1\leq k\leq m$. In this case, $g:=(a_1^{\bf 0},\dots,a_n^{\bf 0})$ is a regular realization of $L$, and $g$ will be a principal realization of $L$ if and only if the minimal leaders of $L$ and $K\langle g\rangle$ coincide.

\begin{lemma}\label{diffspec}
If $f$ is a principal realization and $g$ is a regular realization of the differential kernel $L$, then $g$ is a differential specialization of $f$ (i.e., there is a differential $K$-algebra homomorphism between the differential fields generated by $f$ and $g$ over $K$ mapping $f\mapsto g$).
\end{lemma}
\begin{proof}
Since $f$ is a principal realization of $L$, there is a differential field extension $(M,\dd'_1,\dots,\dd_m')$ of $K$ containing $L=K(a^\xi_i:(\xi,i)\in\Gamma(r))$ such that $\dd'_ka^\xi_i=a^{\xi+{\bf k}}_i$ for all $(\xi,i)\in \Gamma(r-1)$ and $K\langle a_1^{\bf 0},\dots,a_n^{\bf 0}\rangle$ has the same minimal leaders as $L$. Similarly, since $g$ is a regular realization of $L$, there is a differential field extension $(N,\bar\dd_1,\dots,\bar\dd_m)$ of $K$ containing $L$ such that $\bar\dd_ka^\xi_i=a^{\xi+{\bf k}}_i$.  

Now, for each $s\geq r$, the differential kernel given by 
$$
L'_s:=K((\dd')^{\eta}a_i^{\bf 0}:(\eta,i)\in\Gamma(s))
$$
is a generic prolongation of $L$, and the one given by 
$$
\bar L_s:=K(\bar\dd^{\eta}a_i^{\bf 0}:(\eta,i)\in\Gamma(s))
$$
is a prolongation of $L$. By Lemma \ref{prolon}, $\bar L_s$ is a specialization of $L'_s$. Since this holds for all $s\geq r$, the desired differential specialization is obtained by taking the union of this chain.
\end{proof}

\begin{remark}
One can similarly define prolongations, and regular and principal realizations, if the differential kernel is of the form $K(a^\xi_i : (\xi,i) \lhd (\tau,j))$ for some fixed $(\tau,j) \in \N^m \times {\n}$.  In addition, Lemmas \ref{prolon} and \ref{diffspec} also hold in this case, with the same proofs.
\end{remark}

In the ordinary case, $m=1$, every differential kernel has a regular realization (in fact a principal one) \citep[Proposition 3]{Lando}. However, if $m > 1$, regular realizations do not always exist. Moreover, as the following example shows, there are differential kernels of length $r$ with a prolongation of length $2r-1$ but with no regular realization.

\begin{example}\label{examr}
Working with $m=2$ and $n=1$, set $K=\Q$. Let 
$$
L=\Q(a^{(i,j)}:i+j\leq r)
$$
where the $a^{(i,j)}$'s are all algebraically independent over $\Q$ except for the algebraic relations $a^{(0,r)}=a^{(0,r-1)}$ and $a^{(r,0)}=(a^{(0,r-1)})^2$. Set $t:=a^{(0,r-1)}$, so $a^{(0,r)}=t$ and $a^{(r,0)}=t^2$. The field $L$ is a differential kernel over $\Q$ of length $r$. Moreover, it has a (generic) prolongation of length $2r-1$. However, it does not have a prolongation of length $2r$ (and consequently no regular realization of $L$ exists). Indeed, if $L$ had a prolongation 
$$
L'=\Q(a^{(i,j)}:i+j\leq 2r)
$$
with derivations $D_1'$ and $D_2'$, then, as $D_2'(t)=t$, we would have 
$$
D'_1a^{(i,r)}=a^{(i+1,r-1)}
$$
for $0\leq i\leq r-1$, and 
$$
a^{(r,j)}=(D_2')^j(t^2)=2^jt^2
$$
for $1\leq j\leq r-1$. In particular, 
$$
D_1'a^{(r-1,r)}=a^{(r,r-1)}=2^{r-1}t^2 \; \text{ and } \; D_2'a^{(r,r-1)}=D_2'(2^{r-1}t^2)=2^{r}t^2.
$$
This would yield the contradiction
$$2^{r-1}t^2=D_1'a^{(r-1,r)}=a^{(r,r)}=D_2'a^{(r,r-1)}=2^rt^2.$$
\end{example}

Nonetheless, there are conditions on the minimal leaders of a differential kernel that guarantee the existence of a regular realization. In \citep{Pierce}, Pierce proved results of this type using different terminology: In his paper differential kernels are referred to as fields satisfying the \emph{differential condition}, and a regular realization of $L$ is referred to as the existence of a differential field extension of $K$ \emph{compatible} with $L$. Using the terminology of differential kernels \citep[Theorem 4.3]{Pierce} translates to:

\begin{theorem}\label{Pierce43}
Let $L=K(a^{\xi}_i:(\xi,i)\in\Gamma(r))$ be a differential kernel over $K$ for some even integer $r>0$. Suppose further that

\smallskip
\begin{enumerate}
\item[$($\namedlabel{condition1}{$\dagger$}$)$] for every minimal leader $(\xi,i$) of $L$ we have that $\deg\xi\leq \frac{r}{2}$.
\end{enumerate}
\smallskip
Then the differential kernel $L$ has a regular realization.
\end{theorem}

Note that a differential kernel $L$ has a regular realization if and only if it has prolongations of any length. Thus the natural question to ask is: Is the existence of a regular realization guaranteed as long as one can find prolongations of a certain (finite) length? And if so, how can one compute this length, and what is the complexity of this length in terms of the data of the differential kernel? To answer these questions we will need the following terminology. Given an increasing function $f:\N_{>0}\to\N$, we say that $f$ bounds the degree growth of a sequence $\alpha_1,\dots,\alpha_k$ of elements of $\Nn$ if $\deg \alpha_i\leq f(i)$, for $i=1,\dots,k$. We let $\L_{f,m}^n$ be the maximal length of an antichain sequence of $\Nn$ of degree growth bounded by $f$. The existence of the number $\L_{f,m}^n$ follows from generalizations of Dickson's lemma \citep{figueira2011}. Recently, in \citep{LeOv}, an algorithm that computes the exact value of $\L_{f,m}^n$ was established (in fact, an antichain sequence of degree growth bounded by $f$ of maximal length was built).

The following is a consequence of Theorem~\ref{Pierce43} (for details see the proof of \citep[Theorem 4.10]{Pierce} or the discussion after Fact 3.6 of \citep{FreLe}).

\begin{theorem}\label{Pierce410}
Suppose $L = K(a^\xi_i : (\xi,i) \in \Gamma(r))$ is a differential kernel over $K$. Let $f:\N_{>0}\to\N$ be defined as $f(i)=2^i r$. If $L $ has a prolongation of length $2^{\L_{f,m}^n+1}r$, then $L$ has a regular realization.
\end{theorem}

The above theorem motivates the following definition:

\begin{definition}\label{thebignumber}
Given integers $m,n>0$ and $r\geq 0$, we let $T_{r,m}^n$ be the smallest integer $\geq r$ with the following property: For any differential field $(K,\dd_1,\dots,\dd_m)$ of characteristic zero with $m$ commuting derivations and any differential kernel $L$ over $K$ of length $r$, if $L$ has a prolongation of length $T_{r,m}^n$, then $L$ has a regular realization. 
\end{definition}

Theorem \ref{Pierce410} shows that 
$$
T_{r,m}^n\leq 2^{\L_{f,m}^n+1}r \; \text{ where} \; f(i)=2^ir.
$$
This upper bound of $T_{r,m}^n$ is not sharp. For instance, \citep[Proposition 3]{Lando} shows that $T_{r,1}^n=r$, while $2^{\L_{f,1}^n+1}r=2^{n+1}r$. Also, by Lemma~\ref{basic}(3) below we have that $T_{r,2}^1=2r$, while $2^{\L_{f,2}^1+1}r=2^{2r+2}r$. 

In general, for $m>1$, a formula that computes the value of $T_{r,m}^n$ has not yet been found, and thus establishing computationally practical upper bounds is an important problem. In the following sections we establish a much better upper bound for $T_{r,m}^n$ (which is computationally practical for $m\leq 3$). More precisely, in Section \ref{algorithm} we prove that 
$$
T_{r,m}^n\leq \L_{g,m}^n+r-n,
$$
where $g(i)=r+i-1$. One can actually replace $g$ for a \emph{slightly smaller} function $g_n$ (see Section \ref{todefgn}), but the definition of $g_n$ is more convoluted. So, here we decided to state the upper bound in terms of $g$ for the sake of clarity. (Note that $\L_{g_n,m}^n \leq \L_{g,m}^n$ since $g_n(i)\leq g(i)$ for all $i\in \N_{> 0}$.)

To put our new bound in comparison with what was previously known, let us consider some cases:

\begin{enumerate}
\item For $m=1$, our bound reduces to $r$, which, as we pointed out above, is the exact value of $T_{r,1}^n$.
\item For $m = 2$, the previous bound yields
$$
T_{r,2}^n \leq 2^{b_n+1}r,
$$
where $b_n$ is given recursively by $b_0 = 0$ and $b_{i+1} = 2^{b_i+1}r + b_i + 1$; see \citep[\S3]{LeOv}. In particular, 
$$
T_{r,2}^1\leq 2^{2r+2}r \; \text{ and } \; T_{r,2}^2\leq 2^{2^{2r+2}r+2r+3}r.
$$
On the other hand, our new bound (see Theorem \ref{UpperBound2}) yields
$$
T_{r,2}^n\leq 2^nr,
$$
which is a new and practical result.

\item For $m=3$, up until now it was only known that 
$$
T_{1,3}^1\leq 2^{71}\text{ and } T_{2,3}^1\leq 2^{2^{2^{520}+520}+2^{520}+521};
$$
see \citep[Example 3.15]{LeOv}. Our bound (see Corollary \ref{specificvalues}) yields
$$
T_{r,3}^1\leq3(2^r-1).
$$

\item So far no feasible upper bound was known for $m\geq 4$. Our bound yields
$$
T_{1,4}^1\leq 5, \quad T_{1,5}^1\leq 13, \; \text{ and } \; T_{1,6}^1\leq 65533.
$$

\item More generally (for arbitrary $m$), in \citep{LeOv}, it was shown that 
\begin{equation}\label{usingA}
T_{r,m}^n< \begin{cases}
2A(m+3,4r-1)  & \text{ when } n = 1 \\
\frac{2}{n}A(m+5,4nr-1) & \text{ when } n>1. \\
\end{cases}
\end{equation}
Here $A:\N\times \N\to \N$ denotes the Ackermann function:
$$
A(x,y) = \begin{cases} y + 1 & \text{ if } x = 0 \\ A(x-1,1) & \text{ if } x > 0 \text{ and } y = 0 \\ A(x-1,A(x,y-1)) & \text{ if } x,y > 0. \end{cases}
$$
The Ackermann function is known to have extremely large growth, especially in the first input.  For example, $A(1,y) = y+2, \; A(2,y) = 2y+3, \; A(3,y) = 2^{y+3}-3$, and 
$$A(4,y) = \underbrace{{2^2}^{{\cdot}^{{\cdot}^{{\cdot}^2}}}}_{y+3} - 3.$$
Thus, the upper bounds (\ref{usingA}) are not computationally feasible, since the first input is $m+3$ when $n = 1$, and $m+5$ when $n > 1$. On the other hand,  by Corollary \ref{generalAckermannbound},  our bound implies that
$$
T_{r,m}^n \leq A_n(m,r), 
$$
where $A_n:\N \times \N_{>0} \to \N$ is an iterated Ackermann function given by
$$
A_n(x,y)=\begin{cases} A(x,y-1)-1 & \text{ if } n = 1 \\ A(x,A_{n-1}(x,y)-1)-1 & \text{ if } n>1. \end{cases}
$$
This new upper bound is much easier to work with, especially for small inputs.  For example, $A_n(3,r)$ is a tower of exponentials in $r$, where the height of the tower is equal to $n$.
\end{enumerate}


\section{On the existence of principal realizations}\label{expri}

In this section we give an improvement of Theorems \ref{Pierce43} and \ref{Pierce410}. This improvement comes from replacing condition \eqref{condition1} by a \emph{weaker} condition that guarantees the existence of a principal realization of a given differential kernel. We use the notation of the previous section; in particular, $(K,\dd_1,\dots,\dd_m)$ is our base differential field of characteristic zero with $m$ commuting derivations.

Given two elements $\eta = (u_1,\dots,u_m)$ and $\tau = (v_1,\dots,v_m)$ in $\N^m$, we let 
$$
\LUB(\eta,\tau) = (\max(u_1,v_1),\dots,\max(u_m,v_m))
$$ 
be the \emph{least upper bound} of $\eta$ and $\tau$ with respect to the order $\leq$. Given an antichain sequence $\bar \alpha$ of $\Nn$ we let 
$$
\gamma(\bar\alpha)=\{(\LUB(\eta,\tau),i): \, \eta \neq \tau \text{ with } (\eta,i),(\tau,i) \in\bar\alpha \text{ for some }i\}.
$$
Clearly, if for some integer $r\geq 0$ we have $\bar \alpha\subseteq\Gamma(r)$, then $\gamma(\bar \alpha)\subseteq \Gamma(2r)$. For a field extension of $K$ of the form $L=K(a_i^\xi:(\xi,i)\in\Gamma(r))$, we let $\gamma(L)$ denote $\gamma(\bar\alpha)$ where $\bar\alpha=(\alpha_1,\dots,\alpha_k)$ is the antichain sequence consisting of all minimal leaders of $L$ ordered increasingly with respect to $\unlhd$. Note that $$\gamma(L)\subseteq \Gamma(2r).$$

In the proof of Theorem \ref{ulti} below we will use the following fact about extending pairs of commuting derivations. It appears in \citep[Lemma 4.2]{Pierce}.

\begin{lemma}\label{Pierce42}
Suppose a field $M$ has two subfields $L_1$ and $L_2$ with a common subfield $K$.  Suppose also there exist derivations $D_i:L_i\to M$ for $i=1,2$ such that $D_1(K) \subseteq L_2$ and $D_2(K) \subseteq L_1$. If these derivations commute on $K$, then, for any $a\in M$ algebraic over $K$, they extend uniquely to derivations $D_1':L_1(a)\to M$ and $D_2':L_2(a)\to M$, with $D_1'(K(a)) \subseteq L_2(a)$ and $D_2'(K(a)) \subseteq L_1(a)$, which commute on $K(a)$. 
\end{lemma}

\begin{theorem}\label{ulti}
Let $L=K(a^\xi_i:(\xi,i)\in\Gamma(r))$ be a differential kernel over $K$. Suppose further that 

\smallskip
\begin{enumerate}
\item[$($\namedlabel{condition2}{$\sharp$}$)$] For every $(\tau,l)\in\gamma(L)\setminus\Gamma(r)$ and $1\leq i< j\leq m$ such that $(\tau-{\bf i},l) \; \text{ and } \; (\tau -{\bf j},l)$ are leaders, there exists a sequence of minimal leaders $(\eta_1,l),\dots,(\eta_s,l)$ such that $\eta_\ell\leq \tau-{\bf k_\ell}$, with $k_1=i$, $k_s=j$  and some $k_2,\dots,k_{s-1}$, and
\begin{equation}\label{propeq}
\deg \LUB(\eta_\ell,\eta_{\ell+1})\leq r \quad \text{ for } \ell=1,\dots,s-1.
\end{equation}
\end{enumerate}
\smallskip
Then the differential kernel $L$ has a principal realization.
\end{theorem}

\begin{remark} \
\begin{enumerate}
\item One can check that condition \eqref{condition1} of Theorem \ref{Pierce43} implies condition \eqref{condition2}. On the other hand, if $m = 2$, $n=1$, $r=2$, and the only minimal leader of $L$ is $(2,0)$, then condition \eqref{condition1} does \emph{not} hold; however, condition \eqref{condition2} holds trivially. Thus, indeed \eqref{condition2} is a \emph{weaker} condition on the minimal leaders.
\item It is worth pointing out that the converse of Theorem~\ref{ulti} does not generally hold (i.e., \eqref{condition2} is not a necessary condition for the existence of principal realizations). For instance, if $m=2$, $n=1$, $r=1$, and $a^{(1,0)}=a^{(0,1)}=0$, then $L$ has a principal realization but \eqref{condition2} does not hold.
\end{enumerate}
\end{remark}

\begin{proof}
We construct the principal realization recursively. Let $(\tau,l)\in\Nn$ with $\deg\tau>r$.  We want to specify a value for $a^\tau_l$. We assume that we have defined all $a^{\xi}_i$, where $(\xi,i) \lhd (\tau,l)$, such that the field extension
$$
K(a^{\xi}_i: (\xi,i) \lhd (\tau,l))
$$
is a generic prolongation of $L$.

If $(\tau-{\bf i},l)$ is not a leader for all $1\leq i\leq m$, then set $a^\tau_l$ to be transcendental over $K(a^{\xi}_i: (\xi,i) \lhd (\tau,l))$ and define $D_ia_l^{\tau - {\bf i}} := a_l^\tau$. Now, if there is an $i$ such that $(\tau-{\bf i},l)$ is a leader, then the algebraicity of $a_l^{\tau-{\bf i}}$ over $K(a^{\xi}_i:(\xi,i)\lhd (\tau -{\bf i}, l))$ determines what the value of $a^\tau_l$ must be; more precisely, the minimal polynomial of $a^{\tau-{\bf i}}_l$ determines the value $D_ia^{\tau-{\bf i}}_l$, and then we have to set $a^\tau_l:=D_ia^{\tau-{\bf i}}_l$. All we need to check is that if there is another $j$ such that $(\tau-{\bf j},l)$ is a leader, then the value $D_ja^{\tau-{\bf j}}_l$ (determined by the minimal polynomial of $a^{\tau-{\bf j}}_l$) is equal to $D_ia_l^{\tau-{\bf i}}$. This will imply that the value for $a_l^\tau$ is well-defined. 

We now check that indeed $D_ia_l^{\tau-{\bf i}}=D_ja_l^{\tau-{\bf j}}$. Assume for now that $(\tau,l)\in \gamma(L)$ (the other case will be considered below). Condition \eqref{condition2} guarantees the existence of a sequence of minimal leaders $(\eta_1,l),\dots,(\eta_s,l)$ such that $\eta_\ell\leq \tau-{\bf k}_\ell$, with $k_1=i$, $k_s=j$ and some $k_2,\dots,k_{s-1}$, and satisfying (\ref{propeq}). 

\medskip
\noindent {\bf Claim.} For every $1\leq \ell\leq s-1$, we have $D_{k_\ell} a_l^{\tau-{\bf k_\ell}}= D_{k_{\ell+1}}a_l^{\tau-{\bf k}_{\ell+1}}.$

\smallskip
\noindent {\it Proof of Claim.} If $k_\ell = k_{\ell +1}$, then the statement holds trivially. Let $k_{\ell}\neq k_{\ell+1}$ and $\pi=\LUB(\eta_\ell,\eta_{\ell+1})$. By (\ref{propeq}), we have $\deg \pi\leq r< \deg \tau$. In particular, there is $1\leq k\leq m$ such that $\eta_\ell(k)\leq \pi(k) < \tau(k)$, where $\xi(k)$ denotes the $k$-entry of $\xi$. Since $k_\ell\neq k_{\ell+1}$, either $k\neq k_\ell$ or $k\neq k_{\ell+1}$; without loss of generality, we assume that $k\neq k_{\ell}$. We now prove that $(\tau-{\bf k_\ell}-{\bf k},l)$ is a leader. Since $\eta_\ell\leq \tau-{\bf k_\ell}$, $\eta_\ell(k)<\tau(k)$, and $k\neq k_\ell$, we get that $\eta_\ell\leq \tau-{\bf k_\ell}-{\bf k}$. So, since $(\eta_\ell,l)$ is a (minimal) leader, $(\tau-{\bf k_\ell}-{\bf k},l)$ is also a leader. This implies by Lemma~\ref{Pierce42} that the derivations $D_{k_\ell}$ and $D_k$ commute on $a_l^{\tau-{\bf k_\ell}-{\bf k}}$ 
and so
$$
D_{k_\ell}a_l^{\tau-{\bf k_\ell}}=D_{k_\ell}D_ka_l^{\tau-{\bf k_\ell}-{\bf k}}=D_{k}D_{k_\ell}a_l^{\tau-{\bf k_\ell}-{\bf k}} =D_ka_l^{\tau-{\bf k}}.
$$
Now, if $k_{\ell+1}=k$ the result follows from the above equalities. On the other hand, if $k_{\ell+1}\neq k$, we can proceed as before (using the same $k$) to show that $(\tau-{\bf k}_{\ell+1}-{\bf k},l)$ is leader, and thus obtain
$$
D_{k_{\ell+1}}a_l^{\tau-{\bf k}_{\ell+1}}=D_ka_l^{\tau-{\bf k}}.
$$
This proves the claim.

\medskip

It now follows from the claim, since $k_1=i$ and $k_s=j$, that $D_ia_l^{\tau-{\bf i}}=D_ja_l^{\tau-{\bf j}}$, as desired. Now, for the case when $(\tau,l)\notin \gamma(L)$. Let $(\eta_1,l)$ and $(\eta_2,l)$ be any pair of minimal leaders such that $\eta_1\leq \tau-{\bf i}$ and $\eta_2\leq \tau-{\bf j}$. By definition of $\gamma(L)$, we have that $\deg \LUB(\eta_1,\eta_2)< \deg \tau$. One can now proceed as in the proof of the claim, with $\pi=\LUB(\eta_1,\eta_2)$, to show that $D_ia_l^{\tau-{\bf i}}=D_ja_l^{\tau-{\bf j}}$.

One then continues this recursive construction with the tuple succeeding $\tau$ (in the $\lhd$ order). Note that, in each step of this construction, we do not add new minimal leaders, and so the prolongations we obtain at each step still satisfy condition \eqref{condition2} and are generic. By the genericity of each prolongation, this construction yields the desired principal realization of $L$. 
\end{proof}

Let $\bar\alpha=(\alpha_1,\dots,\alpha_k)$ be an antichain sequence of $\Nn$. For each integer $r\geq 0$, let 
$$
\Gamma_{\bar \alpha}(r)=\{\alpha\in\bar\alpha:\, \alpha\in\Gamma(r)\}.
$$
We define $D_{r,\bar\alpha}$ as the smallest integer $p\geq r$ with the following property:

\medskip
\begin{enumerate}
\item[$($\namedlabel{condition3}{$\sharp'$}$)$] For every $(\tau,l)\in\gamma(\Gamma_{\bar\alpha}(p))\setminus\Gamma(p)$ and $1\leq i< j\leq m$ such that $(\tau-{\bf i},l)\geq \beta_1$ and $(\tau -{\bf j},l)\geq \beta_2$ for some $\beta_1,\beta_2\in \Gamma_{\bar\alpha}(p)$, there exists a sequence $(\eta_1,l),\dots,(\eta_s,l)$ in $\Gamma_{\bar\alpha}(p)$ such that $\eta_\ell\leq \tau-{\bf k_\ell}$, with $k_1=i$, $k_s=j$  and some $k_2,\dots,k_{s-1}$, and
\begin{equation}\label{propeq2}
\deg \LUB(\eta_\ell,\eta_{\ell+1})\leq p \quad \text{ for } \ell=1,\dots,s-1.
\end{equation}
\end{enumerate}
\smallskip

Note that if $h\geq r$ is such that $\bar\alpha\subseteq \Gamma(h)$, then $D_{r,\bar\alpha}\leq 2h$. Finally, we set 
$$
C_{r,m}^n := \max\{D_{r,\bar\alpha}: \bar \alpha \text{ is an antichain sequence of } \Nn\}.
$$
In Section \ref{algorithm} we will see that in fact $C_{r,m}^n<\infty$.

\begin{remark}\label{changevar}
Note that, given $r\geq 0$ and an antichain sequence $\bar \alpha$ of $\Nn$, $D_{s,\bar\alpha}=D_{r,\bar\alpha}$ for any $r\leq s\leq D_{r,\bar\alpha}$. 
\end{remark}

\begin{example}
Fix $m=2$ and $n=1$. Consider the case $r=2$ and $\bar\alpha=((2,0),(1,1),(0,2))$.  For all distinct $\xi,\zeta \in \Gamma_{\bar\alpha}(2)$ we have $\deg \LUB(\xi,\zeta) > 2$, so $D_{r,\bar\alpha}\geq 3$.  Note that $\gamma(\Gamma_{\bar\alpha}(3))\setminus \Gamma(3)=\{(2,2)\}$. Setting $\eta_1=(2,0)$, $\eta_2=(1,1)$, $\eta_3=(0,2)$, we see that $\deg\LUB(\eta_\ell,\eta_{\ell+1})\leq 3$ for $\ell=1,2$. This witnesses that $D_{r,\bar\alpha}=3$. In Proposition~\ref{basic}(3), we will see that for any $r$ and $\bar\alpha$ we have $D_{r,\bar\alpha}\leq 2r$.
\end{example}

We can now prove

\begin{theorem}\label{construct}
Let $r$ be a nonnegative integer. Suppose $L = K(a^{\xi}_i: (\xi,i)\in \Gamma(r))$ is a differential kernel over $K$. If $L$ has a prolongation of length $C_{r,m}^n$, then there is some $r\leq h\leq C_{r,m}^n$ such that the differential kernel $K(a^\xi_i:(\xi,i)\in\Gamma(h))$ has a principal realization. In particular, $L$ has a regular realization and so 
$$
T_{r,m}^n\leq C_{r,m}^n.
$$
\end{theorem}
\begin{proof}
Let $\bar\alpha=(\alpha_1,\dots,\alpha_k)$ be the antichain sequence of minimal leaders of the prolongation 
$$
K(a^\xi_i: (\xi,i) \in\Gamma(C_{r,m}^n)).
$$
By definition of $D_{r,\bar\alpha}$ (see property \eqref{condition3} above), if we set $h:=D_{r,\bar\alpha}$, then $h$ has the following three properties:
\begin{enumerate}
\item [(i)] $h\geq r$
\item [(ii)] $C_{r,m}^n\geq h$ and so the field extension $L':=K(a^\xi_i:(\xi,i)\in\Gamma(h))$ is a differential kernel over $K$
\item [(iii)] Since $\Gamma_{\bar\alpha}(h)$ is equal to the set of minimal leaders of $L'$, we have that for every $(\tau,l)\in\gamma(L')\setminus\Gamma(h)$ and $1\leq i<j\leq m$ such that $\tau-{\bf i},\tau-{\bf j}$ are leaders of $L'$, there exists a sequence $(\eta_1,l),\dots,(\eta_s,l)$ of minimal leaders of $L'$ such that $\eta_\ell\leq \tau-{\bf k}_\ell$, with $k_1=i$, $k_s=j$ and some $k_2,\dots,k_{s-1}$, and
$$
\deg \LUB(\eta_\ell,\eta_{\ell+1})\leq h\quad \text{ for } \ell=1,\dots s-1.
$$
\end{enumerate}
Property (iii) is precisely saying that $L'$ satisfies condition \eqref{condition2} of Theorem~\ref{ulti}. Thus, properties (ii) and (iii), together with Theorem \ref{ulti}, yield a principal realization of $L'$. Finally, property (i) implies that this principal realization of $L'$ is a regular realization of $L$.
\end{proof}

\newpage
\begin{remark} \
\begin{enumerate}
\item So far, to the authors' knowledge, there are \emph{no} known cases where $T_{r,m}^n<C_{r,m}^n$. It is thus an interesting problem to determine whether or not these two numbers are equal. Such open questions on the optimality of $C_{r,m}^n$ are part of an ongoing project.
\item It could be interesting to understand how conditions \eqref{condition2} and \eqref{condition3} compare to the Buchberger chain condition (as a refinement of Buchberger's algorithm to compute a Gr\"obner basis). A more detailed analysis of the construction of the bound $C_{r,m}^n$ could potentially lead to improvements in the performance of the Rosenfeld-Gr\"obner algorithm. In this direction we encourage the reader to compare Theorem~\ref{Pierce43} above with \citep[Theorem 3]{Bou} and Theorem~\ref{ulti} above with \citep[Proposition 5]{Bou}.
\end{enumerate}
\end{remark}

In Sections \ref{Macaulay} and \ref{algorithm} we work towards building a recursive algorithm that computes the value of $C_{r,m}^n$. For now, we prove some basic cases.

\begin{proposition}\label{basic} \
\begin{enumerate}
\item $C_{0,m}^n=0$. 
\item For any $r> 0$, $C_{r,1}^n=r$.
\item For any $r>0$, $C_{r,2}^1=2r$. Consequently, by Example \ref{examr}, $T_{r,2}^1=2r$.
\end{enumerate}
\end{proposition}
\begin{proof} \

\noindent (1) This is clear.

\smallskip
\noindent (2) For any antichain sequence $\bar \alpha$ of $\mathbb N\times \n$, condition \eqref{condition3} above is trivially satisfied for any integer $p\geq 0$ since in this case $\gamma(\Gamma_{\bar \alpha}(p))=\emptyset$. Hence, $D_{r,\bar\alpha}= r$, and so $C_{r,1}^n=r$.

\smallskip
\noindent (3) First, to see that $C_{r,2}^1\geq 2r$, consider the antichain sequence $\bar \alpha=((r,0), (0,r))$ of $\N^2$. Since $\gamma(\bar\alpha)=\{\LUB((r,0),(0,r))\}=\{(r,r)\}$, the integer $2r$ satisfies condition \eqref{condition3}, and it is indeed the smallest one as $\bar\alpha$ consists of exactly two elements. Hence, $D_{r,\bar\alpha}=2r$ and so $C_{r,2}^1\geq 2r$.

Now we prove $C_{r,2}^1\leq 2r$. Towards a contradicition assume there is an antichain sequence $\bar \alpha$ of $\N^2$ such that $D_{r,\bar\alpha}> 2r$. First, let us recall a basic fact about blocks of $\N^2$. Recall that a block of $\N^2$ is a subset of the form 
$$
\{(u_1,u_2),(u_1+1,u_2-1),\dots,(u_1+c,u_2-c)\}
$$
for some $u_1,u_2,c\in\N$. Suppose $B$ is a set of elements of $\N^2$ all of degree $d\geq 0$, and let $B'$ be those elements of degree $d+1$ which are $\geq$ some element in $B$. One can check that $|B'|\geq |B|+1$ and $|B'|=|B|+1$ if and only if $B$ is a block.

Now, for each integer $i\geq 0$, we let 
$$\mathcal M_{\bar \alpha}(i)=|\{\xi\in\N^2:\, \deg\xi=i \text{ and } \xi\geq \tau \text{ for some } \tau\in\bar\alpha\}|.$$
Note that $|\mathcal M_{\bar\alpha}(r)|\geq 2$. Indeed, if this were not the case the integer $r$ would satisfy condition \eqref{condition3} and so $D_{r,\bar\alpha}$ would equal $r$, contradicting the fact that $D_{r,\bar\alpha}> 2r$. We now claim that $|\mathcal M_{\bar\alpha}(i+1)|\geq |\mathcal M_{\bar\alpha}(i)|+2$ for $r\leq i<2r $. If this were not the case, then, as we are working in $\N^2$, $|\mathcal M_{\bar\alpha}(i+1)|=|\mathcal M_{\bar\alpha}(i)|+1$. However, as we pointed out above, the latter could only happen if $\mathcal M_{\bar\alpha}(i)$ is a block and $\bar\alpha$ has no elements of degree $i+1$.  But this would imply that the integer $i+1\leq 2r$ satisfies condition \eqref{condition3}, contradicting again the fact that $D_{r,\bar\alpha}>2r$. Putting the previous inequalities together we get $\M_{\bar\alpha}(i)\geq 2(i+1-r)$ for $r\leq i\leq 2r$. In particular, $\mathcal M_{\bar\alpha}(2r)\geq 2r+2$. However, this is impossible since the number of elements of degree $2r$ of $\N^2$ is 
$2r+1$, and so we 
have reached the desired contradiction.
\end{proof}

\section{On Macaulay's theorem}\label{Macaulay}

In this section we prove a key result on the Hilbert-Samuel function that will be used to derive Corollary \ref{HSdeg} below. This will then be used in Section~\ref{algorithm} to provide an algorithm that computes the value of $C_{r,m}^n$. 

Recall 
that we denote $\N^m$ equipped with the product order by $(\N^m,\le)$, and we denote $\N^m$ equipped with the (left) degree-lexicographic order by by $(\N^m,\unlhd)$. Let us start by recalling some basic notions (for details we refer the reader to \citep[Chap.4, \S2]{BrunsHerzog}). A subset $M$ of $\N^m$ is said to be \emph{compressed} if whenever $\xi,\eta\in \N^m$ and $\deg\xi=\deg\eta$ we have
$$
\left(\xi\in M \textrm{ and } \xi\lhd \eta\right)\implies \left( \eta\geq\zeta \text{ for some } \zeta \in M\right).
$$
On the other hand, if $d$ is a positive integer, $M$ is said to be a \emph{$d$-segment} of $\N^m$ if all the elements of $M$ have degree $d$ and, given $\xi,\eta \in \N^m$ with $\xi\lhd\eta$, if $\xi \in M$ then $\eta\in M$. We note that if $M$ is compressed and 
$$
N:=\{\xi\in\N^m:\deg\xi=d \text{ and }\xi\geq\zeta \text{ for some } \zeta\in M\},
$$
then $N$ is a $d$-segment of $\N^m$.

Given positive integers $a$ and $d$, one can write $a$ uniquely in the form
\begin{equation}\label{binrep}
a=\binom{k_d}{d}+\binom{k_{d-1}}{d-1}+\cdots+\binom{k_j}{j},
\end{equation}
where $k_d>k_{d-1}>\cdots > k_j\geq j\geq 1$ for some $j$. One refers to (\ref{binrep}) as the \emph{$d$-binomial representation} of $a$. Now define
$$
a^{\l d\r}:=\binom{k_d+1}{d+1}+\binom{k_{d-1}+1}{d}+\cdots +\binom{k_j+1}{j+1},
$$
and $0^{\l d\r}:=0$.  From \citep[Lemma 4.2.7]{BrunsHerzog}, we have the following property
\begin{equation}\label{sim}
a<b \implies a^{\l d\r} < b^{\l d\r}.
\end{equation}

We now recall the Hilbert-Samuel function: Given any subset $M$ of $\N^m$, we let $H_M:\N\to\N$ be defined as
$$
H_M(d)=|\{\xi\in \N^m:\, \deg\xi=d \textrm{ and } \xi\ngeq \eta \textrm{ for all }\eta\in M\}|.
$$
Macaulay's theorem on the Hilbert-Samuel function states the following (for a proof see Corollary 4.2.9 and Theorem 4.2.10(c) from \citep{BrunsHerzog}).

\begin{theorem}\label{Mac}
For any subset $M$ of $\N^m$, and $d>0$, we have that
$$
H_M(d+1)\leq H_M(d)^{\l d\r}.
$$
Moreover, if $M$ is compressed and $M\subseteq \Gamma(d)$ (i.e., $\deg\xi\leq d$ for all $\xi\in M$), then
$$
H_M(d+1)=H_M(d)^{\l d\r}.
$$
\end{theorem}

We will also make use of the function $S_M$ which is complementary to the Hilbert-Samuel function; that is, for any subset $M$ of $\N^m$, $S_M:\N\to\N$ is given by 
$$
S_M(d) = |\{\xi \in \N^m : \deg \xi = d \text{ and } \xi \geq \eta \text{ for some } \eta \in M\}|.
$$
Note that 
\begin{equation}\label{relSM1}
S_M(d)+H_M(d)=|\{\xi\in\N^m:\, \deg\xi=d\}|=\binom{m-1+d}{d}.
\end{equation}

For any $M\subseteq \N^m$, we define $(1,\dots,m)\cdot M$ to be the set containing all $m$-tuples of the form $(u_1,\dots,u_j+1,\dots,u_m)$ with $(u_1,\dots,u_m) \in M$ and $j = 1,\dots,m$. More generally, for a sequence of integers $1\leq i_1<\cdots <i_s\leq m$, we let $(i_1,\dots,i_s)\cdot M$ be the set $(u_1,\dots,u_{i_j}+1,\dots,u_m)$ with $(u_1,\dots,u_m)\in M$ and $j=1,\dots,s$. We now recall Macaulay's function $a^{(m)}$. For integers $a\geq 0$ and $d>0$, with $a\leq |\{\xi\in \N^m:\deg\xi=d\}|$, we let 
\begin{equation}\label{vala}
a^{(m)}:=|(1,\dots,m)\cdot N_{a,d}|=S_{N_{a,d}}(d+1),
\end{equation}
where $N_{a,d}$ is the subset of $\N^m$ consisting of the $a$ largest elements of $\Gamma(d)$ with respect to $\unlhd$. Note that, by our assumption on $a$ and $d$, the set $N_{a,d}$ is a $d$-segment of $\N^m$ (as defined above); in particular, it is compressed. To justify our notation in (\ref{vala}), we must show that the value $a^{(m)}$ is independent of $d$. To that end, let $d'=d+p$, for a positive integer $p$. Clearly, $N_{a,d'}=(1)^p\cdot N_{a,d}$, where the latter denotes the set of $(u_1+p,\dots,u_m)$ with $(u_1,\dots,u_m)\in N_{a,d}$. Then we have 
$$
(1,\dots,m)\cdot N_{a,d'}=(1)^p\cdot\left( (1,\cdots,m)\cdot N_{a,d}\right),
$$
and hence $|(1,\dots,m)\cdot N_{a,d'}|=|(1,\dots,m)\cdot N_{a,d}|$, as desired.

As a consequence of the moreover clause of Macaulay's theorem (Theorem \ref{Mac}), for integers $b\geq 0$ and $d>0$, with $b\leq |\{\xi\in \N^m:\deg\xi=d\}|=\binom{m-1+d}{d}$, we have that
$$
b^{\l d\r}=|\{\xi\in\N^m: \deg \xi=d+1 \text{ and } \xi\notin (1,\dots,m)\cdot N_{a,d}\}|=H_{N_{a,d}}(d+1),
$$
where $a:=\binom{m-1+d}{d}-b$. This implies that $b^{\l d\r}=\binom{m+d}{d+1}-a^{(m)}$; in particular, for any $M$ we have
\begin{equation}\label{relSM2}
H_M(d)^{\l d\r}=\binom{m+d}{d+1} - S_M(d)^{(m)}.
\end{equation}
Thus, with the above notation, Theorem \ref{Mac} can be reformulated as

\begin{corollary}\label{varMac}
For any subset $M$ of $\N^m$, and $d>0$, we have that
$$
S_M(d+1)\geq S_M(d)^{(m)}.
$$
Moreover, if $M$ is compressed and $M\subseteq \Gamma(d)$, then
$$
S_M(d+1)=S_M(d)^{(m)}.
$$
\end{corollary}
\begin{remark}
The formulation of this corollary is quite similar to how Macaulay originally presented his theorem in the 1920s (see \citep{Macaulay} or \citep{Sperner}). 
\end{remark}
\begin{proof}
By (\ref{relSM1}), (\ref{relSM2}) and Theorem \ref{Mac}, we have
\begin{align*}
S_M(d+1)= & \; \binom{m+d}{d+1}-H_M(d+1) \\
\geq & \; \binom{m+d}{d+1}-H_M(d)^{\l d\r} \\
= & \; S_M(d)^{(m)}.
\end{align*}
For the moreover clause one simply replaces the above inequality by equality.
\end{proof}

We now fix some notation that will be used in the proof of Theorem \ref{strictmacaulay} below. 

\begin{definition}\label{connected}
Let $d$ be a nonnegative integer and $M$ a subset of $\N^m$. Given $\tau\in\N^m$ of $\deg\tau> d+1$, and distinct $\xi,\, \zeta\in M\cap \Gamma(d)$ both $<$ $\tau$ (recall that $<$ denotes the product order of $\N^m$), we say that $\xi$ and $\zeta$ are \emph{$\tau_{d,M}$-connected} if there is a sequence of elements $\eta_1,\dots,\eta_s$ of $M\cap\Gamma(d)$ all $<$ $\tau$ with $\eta_1=\xi$, $\eta_s=\zeta$, and such that for all $1 \leq i \leq s-1$,
$$
\deg \LUB(\eta_i,\eta_{i+1})\leq d+1.
$$
\end{definition}

\smallskip
Given an integer $d\geq 0$, consider the following condition on $M\subseteq \N^m$:

\begin{enumerate}
\item[$($\namedlabel{condition4}{$*$}$)$] There exist two distinct elements $\xi, \zeta \in M \cap \Gamma(d)$ 
that are \emph{not} $LUB(\xi, \zeta)_{d,M}$-connected.
\end{enumerate}

In other words (or more explicitly), condition \eqref{condition4} says that there are two distinct elements $\xi,\zeta \in M\cap \Gamma(d)$ such that for every sequence $\eta_1,\dots,\eta_s$ of elements of $M\cap\Gamma(d)$ all $<$ $\LUB(\xi,\zeta)$ with $\eta_1 = \xi$ and $\eta_s = \zeta$, there exists $1 \leq i \leq s-1$ such that $\deg\LUB(\eta_i,\eta_{i+1}) > d+1.$

\begin{example}\label{constar}
Let $m=3$ and $M=\{\xi\in \N^3: \deg\xi=2\}\setminus\{(1,1,0)\}$. We claim that $M$ satisfies condition \eqref{condition4} with $d=2$. The two elements witnessing this are $\eta=(2,0,0)$ and $\zeta=(0,2,0)$. Indeed, the only elements of $M$ that are $<\LUB(\eta,\zeta)=(2,2,0)$ are $\eta$ and $\zeta$. In other words, $(2,0,0)$ and $(0,2,0)$ are \emph{not} $(2,2,0)_{2,M}$-connected. 
\end{example}

\begin{remark}\label{tau}
Suppose $M$ satisfies condition \eqref{condition4} for a fixed $d$. Then, for any pair of distinct elements $\xi,\zeta \in M\cap \Gamma(d)$ given as in condition \eqref{condition4}, we have that $\deg\LUB(\xi,\zeta) > d+1$.  Hence, $M\cap\Gamma(d)$ contains two distinct elements $\xi$ and $\zeta$  that are not $\LUB(\xi,\zeta)_{d,M}$-connected.  Moreover, such a pair $(\xi,\zeta)$ can be chosen with the following additional property: for any two distinct elements $\eta,\pi \in M\cap\Gamma(d)$ both $< \LUB(\xi,\zeta)$, either $\eta$ and $\pi$ are $\LUB(\xi,\zeta)_{d,M}$-connected, or $\LUB(\eta,\pi) = \LUB(\xi,\zeta)$. To see this, suppose there exist distinct $\xi',\zeta' \in M\cap\Gamma(d)$ both $< \LUB(\xi,\zeta)$ that are not $\LUB(\xi,\zeta)_{d,M}$-connected but $\LUB(\xi',\zeta') \neq \LUB(\xi,\zeta)$. Then $\LUB(\xi',\zeta')<\LUB(\xi,\zeta)$. In this case, we replace the pair $(\xi,\zeta)$ with the pair $(\xi',\zeta')$. This process will eventually produce the desired pair (after finitely many steps, 
since at each step the 
degree of $\LUB(\xi,\zeta)$ decreases).
\end{remark}

We will need the following technical result of Sperner on the Macaulay function (see \citep[\S3, p.196]{Sperner}).

\begin{lemma}\label{techSper}
Let $A,B,C$ be nonnegative integers. If $A>0$, $A=B+C$, and $C^{(m-1)}<A^{(m)}-A$, then 
$$
B^{(m)}+C^{(m-1)}\geq A^{(m)}.
$$
\end{lemma}

We are finally ready to prove the main theorem of this section, which can be regarded as the key result of the paper.

\begin{theorem}\label{strictmacaulay}
Let $d>0$ be an integer and $M$ a subset of $\N^m$.  If $M$ satisfies condition \eqref{condition4} above, then we have the following strict inequality
$$
H_M(d+1) < H_M(d)^{\l d\r}.
$$
\end{theorem}

\begin{proof}
We first make some simplifications. By definition of the Hilbert-Samuel function, we have
$$H_M(d)=H_N(d) \,\text{ and }\, H_M(d+1)\leq H_N(d+1),$$
where $N=\{\xi\in\N^m: \deg \xi=d \text{ and } \xi\geq \eta \text{ for some }\eta\in M\}$, and so it would suffice to prove the theorem for $N$. Hence, we may (and do) assume that all the elements of $M$ have degree $d$. 

Note that the desired inequality is equivalent to
\begin{equation}\label{alter}
S_M(d+1)>S_M(d)^{(m)}.
\end{equation}
Indeed, if (\ref{alter}) holds, by (\ref{relSM1}) and (\ref{relSM2}), we would have
$$
H_M(d+1)=  \; \binom{m+d}{d+1}-S_M(d+1) <  \; \binom{m+d}{d+1}-S_M(d)^{(m)} =  \; H_M(d)^{\l d\r}.
$$
Thus, it suffices to prove (\ref{alter}). Note that, by our assumption that all the elements of $M$ have degree $d$, we have $|M| = S_M(d)$ and $|(1,\dots,m)\cdot M| = S_M(d+1)$. 

Now let $(\xi,\zeta)$ be a pair of distinct elements of $M$ as in Remark \ref{tau} and set 
$$
\tau:=\LUB(\xi,\zeta)=(v_1,\dots,v_m);
$$
in other words, $\xi = (a_1,\dots,a_m)$ and $\zeta = (b_1,\dots,b_m)$ are elements of $M$ that are not $\tau_{d,M}$-connected, and for any two distinct elements $\eta,\pi \in M$ both $< \tau$, either $\eta$ and $\pi$ are $\tau_{d,M}$-connected or $\LUB(\eta,\pi) = \tau$. We assume that $a_1<b_1$; if not simply permute the variables.

Let $A := S_M(d)=|M|$ and $F:= S_M(d+1)=|(1,\dots,m)\cdot M|$.  Thus, we must show that 
$$
F>A^{(m)}.
$$ 
We prove the result by induction on $A$.  Since $M$ has at least two elements, the base case is $A = 2$ and so $M = \{\xi,\zeta\}$.  In this case, $A^{(m)} = 2m-1$, and saying that $\xi$ and $\zeta$ are not $\tau_{d,M}$-connected is equivalent to saying that $\deg\tau > d+1$.  But then we cannot have $\xi + {\bf i} = \zeta + {\bf j}$ for any $1\leq i,j\leq m$, so $F = 2m > A^{(m)}$.

Now we prove the induction step, and so assume $A\geq 3$.  Let $(u_1,\dots,u_m)$ be the least element of $M$ with respect to the (left) degree-lexicographical order $\unlhd$. We can then write
$$
M = M_0 \cup M_1,
$$
where $M_0$ consists of all elements $(t_1,\dots,t_m) \in M$ with $t_1 > u_1$, and $M_1$ consists of all elements $(t_1,\dots,t_m) \in M$ with $t_1 = u_1$.  Note that $M_0 \cap M_1 = \emptyset$.  We then have the following inclusions:
\begin{equation}\label{incl1}
(1)\cdot M \,\cup\, (2,\dots,m)\cdot M_1 \,\subseteq\, (1,\dots,m)\cdot M,
\end{equation}
\begin{equation}\label{incl2}
(1,\dots,m)\cdot M_0\, \cup\, (2,\dots,m)\cdot M_1\, \subseteq \, (1,\dots,m)\cdot M.
\end{equation}
In addition we have that 
$$
(1)\cdot M\, \cap\, (2,\dots,m)\cdot M_1 = \emptyset \,\text{ and } \, (1,\dots,m)\cdot M_0\, \cap\, (2,\dots,m)\cdot M_1 = \emptyset.
$$

We now prove that, under our assumptions, the inclusion \eqref{incl1} is strict.  First note that all tuples $\pi = (c_1,\dots,c_m) \in M$ such that $\pi<\tau$ and $c_1 = a_1<b_1$ are $\tau_{d,M}$-connected to $\xi$, otherwise this would contradict the choice of $\tau$ as $\LUB(\xi,\pi) \neq \tau$.  Let $a$ be the smallest integer with $a>a_1$ and such that there is $\pi=(c_1,\dots,c_m) \in M$ with $\pi<\tau$, not $\tau_{d,M}$-connected to $\xi$, and $c_1=a$.  Note that $a \leq b_1$. Also, note that there is $1<i\leq m$ such that $c_i<v_i$ (if not, we would have $\pi>\xi$). Set
$$
\rho=(c_1-1,c_2,\dots,c_i+1,\dots,c_m).
$$
Then, $\rho$ is not in $M$. Indeed, it it were, then $\pi$ and $\rho$ would be $\tau_{d,M}$-connected and $\rho$ and $\xi$ would also be $\tau_{d,M}$-connected (by construction of $\tau$), so $\pi$ would be $\tau_{d,M}$-connected to $\xi$.  This shows that $\pi+{\bf i}\in (1,\dots,m)\cdot M$ but $\pi+{\bf i}\notin  (1)\cdot M\,\cup\, (2,\dots,m)\cdot M_1$, as desired.

Now we prove that if $M_0$ does not satisfy condition \eqref{condition4}, then containment \eqref{incl2} is strict. In this case, we must have $\xi\in M_1$. Also, note that for every $1< i\leq m$ such that $a_i>0$, if we set
$$
\nu=(a_1+1,a_2,\dots,a_i-1,\dots,a_m),
$$ 
then $\nu<\tau$ but it cannot be in $M_0$. Indeed, if it were, then $\nu$ and $\zeta$ would witness that $M_0$ satisfies condition \eqref{condition4} since $\xi$ and $\nu$ are $\tau_{d,M}$-connected. This shows that $\xi+{\bf 1}\in (1,\dots, m)\cdot M$ but $\xi+{\bf 1}\notin (1,\dots,m)\cdot M_0\, \cup\, (2,\dots,m)\cdot M_1$, as desired.

Let $B = |M_0|$ and let $C = |M_1|$.  Since we have shown that inclusion \eqref{incl1} is strict, an application of Corollary \ref{varMac} yields
$$
F> A+C^{(m-1)}.
$$
On the other hand, if inclusion \eqref{incl2} is strict, another application of Corollary \ref{varMac} yields
$$
F>B^{(m)}+C^{(m-1)}.
$$
Finally, if \eqref{incl2} is an equality, then we have shown that $M_0$ must satisfy \eqref{condition4}. Since $B < A$ (as $M_0 \subsetneq M$), by induction we have that in this case $|(1,\dots,m)\cdot B|>B^{(m)}$, and so
$$
F>B^{(m)}+C^{(m-1)}.
$$
Therefore, we always have that
\begin{equation}\label{equa}
F>A+C^{(m-1)}\quad \text{ and }\quad F>B^{(m)}+C^{(m-1)}.
\end{equation}
If $C^{(m-1)}\geq A^{(m)}-A$, then it follows from the first inequality of (\ref{equa}) that $F>A^{(m)}$. For the remaining case $C^{(m-1)}<A^{(m)}-A$, since $A\geq 3$ and $A=B+C$, Lemma \ref{techSper} yields $B^{(m)}+C^{(m-1)}\geq A^{(m)}$. It now follows from the second inequality of (\ref{equa}) that $F>A^{(m)}$. This concludes the proof.
\end{proof}

\begin{remark}\label{correctex} \
\begin{enumerate}
\item Theorem \ref{strictmacaulay} seems to be of independent interest. It states that a necessary condition for the Hilbert-Samuel function of $M$ to have maximal growth at $d+1$ is that every pair $\xi,\zeta$ of distinct elements of $M\cap\Gamma(d)$ is $\LUB(\xi,\zeta)_{d,M}$-connected.
\item For $m=2$, the converse of Theorem \ref{strictmacaulay} holds. Indeed, if $M$ is a subset of $\N^2$ all of whose elements have degree $d$, then $H_M(d+1)=H_M(d)^{\l d\r}$ if and only if $M$ is a block (i.e., $M$ is of the form $\{(u_1,u_2),(u_1+1,u_2-1),\dots,(u_1+c,u_2-c)\}$ for some $u_1,u_2,c\in\N$), and if $M$ is a block then $M$ does not satisfy condition \eqref{condition4}. On the other hand, when $m\geq 3$, the converse of Theorem \ref{strictmacaulay} does not generally hold.  For a (counter-)example, consider the case $m=3$ and $M=\{\xi\in \N^3: \deg\xi=3\}\setminus\{(1,1,1)\}$. Then $M$ does not satisfy condition \eqref{condition4} with $d=3$. Let us verify that the points $(2,0,1)$ and $(0,2,1)$ are indeed $(2,2,1)_{3,M}$-connected (it is trivial to check connectedness for other pairs of points in $M$). Let $\eta_1=(2,0,1)$, $\eta_2=(2,1,0)$, $\eta_3=(1,2,0)$ and $\eta_4=(0,2,1)$. Then $\deg\LUB(\eta_i,\eta_{i+1})\leq 4$ for $i=1,2,3$, and so this sequence witnesses the desired connectivity (the reader might want to compare this with Example~\ref{constar}). However, 
$$
H_M(4)=0<1=H_M(3)^{\l 3\r}.
$$
\end{enumerate}
\end{remark}

To finish this section, we want to connect the previous discussion to our work with antichains from previous sections.  Given an antichain sequence $\bar\xi = (\xi_1,\dots,\xi_k)$ of $\N^m$, for each $i \geq 0$ the Hilbert-Samuel function $H_{\bar\xi}^i:\N \to \N$ is defined as
$$
H_{\bar\xi}^i(d) = |\{\eta \in \N^m : \deg \eta = d \text{ and } \eta \not \geq \xi_j \text{ for all } j \leq i \text{ for which } \xi_j \text{ is defined}\}|.
$$
If for each $i\geq 0$ we let 
$$M_{i}=\{\eta\in \N^m:  \eta\geq \xi_j \text{ for some } \xi_j \text{ with } j \leq i\},$$
we see that then $H_{M_{i}}(d) = H_{\bar\xi}^i(d)$.  Hence, a direct consequence of Theorem \ref{strictmacaulay} is the following:

\begin{corollary}\label{HSdeg}
Let $d>1$ be an integer and $\bar\xi=(\xi_1,\dots,\xi_k)$ an antichain sequence of $\N^m$. If $H_{\bar\xi}^k(d)=H_{\bar\xi}^k(d-1)^{\langle d-1\rangle}$, then for each pair $\xi_i\neq \xi_j$, both having degree at most $d-1$, there exists a sequence $\eta_1,\dots,\eta_s$ of distinct elements of $\Gamma_{\bar\xi}(d-1)=\bar\xi\cap\Gamma(d-1)$ all $<\LUB(\xi_i,\xi_j)$ such that $\eta_1=\xi_i$, $\eta_s=\xi_j$, and 
$$
\deg \LUB(\eta_\ell,\eta_{\ell+1})\leq d, \quad \text{ for all } \ell=1,\dots,s-1.
$$
\end{corollary}

\section{An algorithm to compute $C_{r,m}^n$}\label{algorithm}

In this section we prove that there is a recursive algorithm that computes $C_{r,m}^n$. We first deal with the case $n=1$ (Theorem \ref{UpperBound}), and then we prove that for $n>1$ the value is obtained by compositions in the ``$r$'' input (Theorem \ref{UpperBound2}).

Recall from Section \ref{diffkerpre} that for any increasing function $f:\N_{>0}\to\N$ we say that $f$ bounds the degree growth of a sequence $\alpha_1,\dots,\alpha_k$ of $\Nn$ if $\deg \alpha_i \leq f(i)$ for all $i=1,\dots,k$.  Also, $\L_{f,m}^n$ denotes the maximal length of an antichain sequence of $\Nn$ of degree growth bounded by $f$. In \citep{LeOv} an algorithm that computes the value of $\L_{f,m}^n$ was established and an antichain sequence of maximal length was built. We discuss this in more detail below.

\subsection{The case $n=1$} \label{subsec1}

Throughout this subsection we let $g:\N_{>0}\to\N$ be the increasing function defined as $g(1)=r$ and $g(i)=i+r-2$ for $i\geq 2$. We will prove that 
$$
C_{r,m}^1= \L_{g,m}^1+r-1.
$$
In Proposition \ref{basic} we  already proved that this equality holds in the case $r=0$ or $m=1$. We now assume $r\geq 1$ and $m\geq 2$. Note that in this case we have $\L_{g,m}^1\geq 2$, and so the above equality is equivalent to
\begin{equation}\label{forn1}
C_{r,m}^1= g(\L_{g,m}^1)+1.
\end{equation}

Let $\bar\mu = (\mu_1,\dots,\mu_L)$ be the antichain sequence defined recursively as follows:
$$
\mu_1 = \max_{\unlhd}\{\xi \in \N^m : \deg \xi = g(1)\},
$$
and, as long as it is possible, 
$$
\mu_i = \max_{\unlhd}\{\xi \in \N^m : \deg \xi = g(i) \text{ and } \xi \not \geq \mu_1,\dots,\mu_{i-1}\}.
$$

In \citep[\S3.2]{LeOv}, it is shown that $\bar\mu$ is a compressed antichain sequence of $\N^m$ having length $L = \L_{g,m}^1$ (i.e., $\bar \mu$ is of maximal length among antichain sequences of $\N^m$ with degree growth bounded by $g$). It is also observed that $H_{\bar\mu}^L(\deg \mu_L)=H_{\bar\mu}^L(g(L))=0$, where recall that $H_{\bar \mu}^i$ denotes the Hilbert-Samuel function of $\bar \mu$, that is, for $i,d\geq 0$,
$$
H_{\bar\mu}^i(d) = |\{\xi \in \N^m : \deg \xi = d \text{ and } \xi \not \geq \mu_j \text{ for all } j \leq i \text{ for which } \mu_j \text{ is defined}\}|.
$$

The antichain sequence $\bar \mu$ can be more explicitly constructed as follows:
\begin{enumerate}
\item[(i)] if $\mu_i = (u_1,\dots,u_s,0,\dots,0,u_m)$ with $s < m-1$ and $u_s > 0$, then
$$
\mu_{i+1} = (u_1,\dots,u_s-1,g(i+1) - g(i) + u_m+1,0,\dots,0)
$$
\item[(ii)] if $\mu_i = (u_1,\dots,u_{m-1},u_m)$ with $u_{m-1} > 0$, then
$$
\mu_{i+1} = (u_1,\dots,u_{m-1} - 1, g(i+1) - g(i) + u_m+1).
$$
\end{enumerate}

From this recursive construction of $\bar\mu$, one obtains (see \citep[Corollary 3.10]{LeOv})
\begin{equation}\label{algobo}
\L_{g,m}^1=\Psi_{g,m}(1,(g(1),0\dots,0)),
\end{equation}
where $\Psi_{g,m}:\N_{>0}\times\N^{m}\to \N$ is given by
$$
\Psi_{g,m}(i,(0,\dots,0,u_n))=i
$$
with
\begin{align*}
\Psi_{g,m}&(i-1,(u_1,\dots,u_s,0,\dots,0,u_m))\\&=\Psi_{g,m}(i,(u_1,\dots,u_s-1,g(i)-g(i-1)+u_m+1,0,\dots,0)),\quad s<m-1, \ u_s>0
\end{align*}
 and 
\begin{align*}
\Psi_{g,m}&(i-1,(u_1,\dots,u_{m-1},u_m))\\&=\Psi_{g,m}(i,(u_1,\dots,u_{m-1}-1,g(i)-g(i-1)+u_m+1)),\quad u_{m-1}>0. 
\end{align*}

For example, when $m=2$, the sequence $\bar\mu$ is given by
$$
\mu_1=(r,0), \; \mu_2=(r-1,1),\; \mu_3=(r-2,3),\; \mu_4=(r-3,5),\dots,\; \mu_{r+1}=(0,2r-1),
$$
and so $L=\L_{g,2}^1=r+1$.

By the above discussion, it suffices to establish (\ref{forn1}) to prove that there is a recursive algorithm that computes the value of $C_{r,m}^1$. We first prove that $C_{r,m}^1\geq g(\L_{g,m}^1)+1$. 

\begin{proposition}\label{botu}
With $\bar\mu$ as above, we have $D_{r,\bar\mu}=g(\L_{g,m}^1)+1$ (see Section \ref{expri} for the definition of $D_{r,\bar\mu}$). In particular, $C_{r,m}^1\geq g(\L_{g,m}^1)+1$.
\end{proposition}
\begin{proof}
For each $i=1,\dots,L$, we let $\bar \xi_i$ be the antichain sequence $(\mu_1,\dots,\mu_i)$. Recall that $L=\L_{g,m}^1$. It suffices to prove 
\begin{equation}\label{degmuse}
D_{r,\bar\xi_i}=\deg\mu_i+1 \quad \text{ for } i=2,\dots,L.
\end{equation}
Indeed, if (\ref{degmuse}) holds, then taking $i=L$ yields $D_{r,\bar\mu}=\deg\mu_L+1=g(\L_{g,m}^1)+1$. At this point we encourage the reader to look back at the definition of $D_{r,\bar\xi_i}$ in Section~\ref{expri}. 

We now prove (\ref{degmuse}) by induction on $i$. We actually prove a little bit more: in addition to (\ref{degmuse}), we prove that for each pair of distinct elements $\mu_q,\mu_t\in\bar\xi_i$ 
\begin{equation}\label{seqtu}
\text{there are } \eta_1,\dots,\eta_s \in \bar\xi_i \,\text{ all } <\LUB(\mu_q,\mu_t) \text{ such that } \eta_1=\mu_q,\, \eta_s=\mu_t
\end{equation}
$$
\text{ and } \deg\LUB(\eta_\ell,\eta_{\ell+1})\leq \deg\mu_i+1  \,\text{ for } \ell=1,\dots,s-1.
$$

\medskip
For the base case $i=2$, note that
$$
\bar\xi_2=(\mu_1,\mu_2)=((r,0,\dots,0),(r-1,1,0,\dots,0)),
$$
so $\gamma(\bar\xi_2)=\{\LUB(\mu_1,\mu_2)\}=\{(r,1,0\dots,0)\}$. Since $\deg(r,1,0,\dots,0)=r+1$, we get that $D_{r,\bar\xi_2}=r+1=\deg\mu_2+1$.  To show condition (\ref{seqtu}) we simply take $\eta_1=\mu_1$ and $\eta_2=\mu_2$. 

For the induction step we fix $3\leq i \leq L$, and assume $D_{r,\bar\xi_{i-1}}=\deg\mu_{i-1}+1$ and that condition (\ref{seqtu}) holds for $i-1$. Since $\bar\xi_i$ is the concatenation of $\bar\xi_{i-1}$ and $\mu_i$ (with $\deg\mu_i=\deg\mu_{i-1}+1=D_{r,\bar\xi_{i-1}}$), we have that $D_{r,\bar\xi_i}\geq D_{r,\bar\xi_{i-1}}=\deg\mu_i$. It remains to show that $D_{r,\xi_i}\neq \deg\mu_i$, that the integer $\deg\mu_i+1$ satisfies condition \eqref{condition3} of Section~\ref{expri}, and that condition (\ref{seqtu}) holds. To do this we will prove that for any $q<i$ there exists $t<i$ such that
\begin{equation}\label{exis}
\mu_t<\LUB(\mu_q,\mu_i) \text{ and } \deg\LUB(\mu_t,\mu_i)=\deg\mu_i+1,
\end{equation}
and this will complete the proof. Indeed, suppose (\ref{exis}) holds, and set $\zeta=\LUB(\mu_t,\mu_i)\in\gamma(\bar\xi_i)$, where this $t$ is the one associated to $q=1$. Then, there exists $1\leq k\leq m$ such that $\mu_i=\zeta-{\bf k}$, and so there cannot be $p<i$ such that $\mu_p\leq \zeta-{\bf k}$. Thus, this $\zeta$ witnesses the fact that $D_{r,\xi_i}\neq \deg\mu_i$. On the other hand, observe that if condition (\ref{seqtu}) holds then the integer $\deg\mu_i+1$ satisfies condition \eqref{condition3}. Thus, it would be enough to prove condition (\ref{seqtu}). To do this, let $\mu_p\rhd \mu_q$ be a pair of elements of $\bar\xi_i$. If $\mu_p,\mu_q\in\bar\xi_{i-1}$, then, by induction, there is a sequence with the desired properties. So now suppose $p=i$. By (\ref{exis}), there is $\mu_t\in\bar\xi_{i-1}$ such that $\mu_t<\LUB(\mu_p,\mu_q)$ and $\deg\LUB(\mu_t,\mu_{p})\leq \deg\mu_i+1$. Hence, in this case, the desired sequence can be obtained by starting with $\eta_1=\mu_p$, $\eta_2=\mu_t$, and 
then continuing with an appropriate sequence going from $\mu_t$ to $\mu_q$ (which exists by induction).

Finally, we prove (\ref{exis}). To do this, let $q<i$ and consider the two possible shapes that $\mu_i$ can take according to the construction of $\bar \mu$ above:

\medskip
\noindent \underline{Case 1.} Suppose $\mu_{i-1}=(u_1,\dots,u_{m-1},u_m)$ with $u_{m-1}>0$. Then, by construction of $\bar\mu$,
$$
\mu_{i}=(u_1,\dots,u_{m-1}-1,a),\quad \text{ where } a=g(i)-u_1-\cdots-u_{m-1}+1.
$$
Let $\mu_q=(v_1,\dots,v_m)$ and $1\leq l \leq m$ be the smallest integer such that the $l$-entry of $\mu_q$ is strictly larger than the $l$-entry of $\mu_{i}$. Note that we must have $l<m$. Indeed, since $q<i$, the $l$-entry is the first entry (from left to right) where $\mu_q$ and $\mu_{i}$ differ. By construction of $\bar\mu$, we can find $t<i$ such that $\mu_t$ has the form $(u_1,\dots,u_{l-1},w_l,\dots,w_m)$ with $w_l$ equal to $1+($the $l$-entry of $\mu_{i})$, and $w_p$ less than or equal to the $p$-entry of $\mu_{i}$ for $l<p\leq m$. Then $\mu_t<\LUB(\mu_q,\mu_i)$ and
$$
\deg \LUB(\mu_t,\mu_{i})=\deg\mu_i+1.
$$

\medskip
\noindent \underline{Case 2.} Suppose $\mu_{i-1}=(u_1,\dots,u_s,0,\dots,0,u_m)$ with $s<m-1$ and $u_s>0$. Then, by construction of $\bar\mu$, 
$$\mu_{i}=(u_1,\dots,u_s-1,a,0,\dots,0),\quad \text{ where } a=g(i)-u_1-\dots-u_s+1.$$
Let $\mu_q=(v_1,\dots,v_m)$ and $1\leq l\leq m$ be the smallest integer such that the $l$-entry of $\mu_q$ is strictly larger than the $l$-entry of $\mu_{i}$. The same reasoning as in Case 1 yields that $l\leq s$. Again by construction of $\bar\mu$, we can find $t<i$ such that $\mu_t$ has the form $(u_1,\dots,u_{l-1},w_l,\dots,w_{s},w_{s+1},0,\dots,0)$
with $w_l$ equal to $1+($the $l$-entry of $\mu_{i})$, and $w_p$ less than or equal to the $p$-entry of $\mu_{i}$ for $l<p\leq s+1$. Then $\mu_t<\LUB(\mu_q,\mu_i)$ and 
$$
\deg \LUB(\mu_t,\mu_{i})=\deg\mu_i+1.
$$ 
\end{proof}

It remains to show that $C_{r,m}^1\leq g(\L_{g,m}^1)+1$. To do this, suppose there is an antichain sequence $\bar\xi=(\xi_1,\dots,\xi_M)$ of $\N^m$ such that $D_{r,\bar\xi}\geq g(\L_{g,m}^1)+1$. We must show that then $D_{r,\bar\xi}\leq g(\L_{g,m}^1)+1$. 

The following result gives the relationship between the Hilbert-Samuel functions of $\bar\mu$ and $\bar \xi$. This is where Corollary \ref{HSdeg} is used.

\begin{theorem}\label{intheb}
With $\bar\mu$ and $\bar\xi$ as above, we have that 
$$
H_{\bar\xi}^i(d)\leq H_{\bar\mu}^i(d)
$$ 
for all $i, d\geq 0$. As a result, $D_{r,\bar\xi}\leq g(\L_{g,m}^1)+1$. 
\end{theorem}

\begin{proof}
We proceed by induction on $i$. For the base case $i=0$, we have
$$
H_{\bar\xi}^0(d)=\binom{m-1+d}{d}=H_{\bar\mu}^0 (d),
$$
which is the number of $m$-tuples of degree $d$. 

We now proceed with the induction step $i+1$. Note that, since $D_{r,\bar\xi}\geq g(\L_{g,m}^1)+1$, the sequence $\bar\xi$ contains at least two elements of degree at most $r$. It follows then that $H_{\bar\xi}^1(d)\leq H_{\bar\mu}^1(d)$ and $H_{\bar\xi}^2(d)\leq H_{\bar\mu}^2(d)$ for all $d\geq 0$. Thus, we assume that $i\geq 2$. We have that for $d<\deg \mu_{i+1}$, 
$$
H_{\bar\xi}^{i+1}(d)\leq H_{\bar\xi}^{i}(d)\leq H_{\bar\mu}^i(d)=H_{\bar\mu}^{i+1}(d).
$$
Now consider the case when $d = \deg \mu_{i+1}$ (note that $d>1$ since $r>0$ and $i\geq 2$). 

\medskip
\noindent {\bf Claim.} Either $H_{\bar\xi}^{i+1}(d)< H_{\bar\xi}^i(d)$ or $H_{\bar\xi}^i(d)< H_{\bar\mu}^i(d)$.

\smallskip
\noindent {\it Proof of Claim.} Towards a contradiction suppose 
\begin{equation}\label{WMA}
H_{\bar\xi}^{i+1}(d)=H_{\bar\xi}^i(d)=H_{\bar\mu}^i(d). 
\end{equation}
By the induction hypothesis, property \eqref{sim}, and Macaulay's theorem (Theorem \ref{Mac}), we get 
$$
H_{\bar\xi}^i(d-1)^{\langle d-1\rangle}\leq H_{\bar\mu}^i(d-1)^{\l d-1\r}=H_{\bar\mu}^i(d)=H_{\bar\xi}^i(d).
$$ 
By Macaulay's theorem, this inequality implies that $H_{\bar\xi}^i(d)=H_{\bar\xi}^i(d-1)^{\langle d-1\rangle}$. This equality, together with $H_{\bar\xi}^{i+1}(d)=H_{\bar\xi}^i(d)$, implies that $\deg\xi_j\neq d$ for all $j\leq i+1$ for which $\xi_j$ is defined. Thus, by Corollary~\ref{HSdeg}, condition~\eqref{condition3} is satisfied with $p=d$, so $D_{r,\bar\xi} \le d$.  Since $\deg \mu_{i+1} < D_{r,\bar\mu}$ by Proposition~\ref{botu}, we obtain
$$
D_{r,\bar\xi}\leq d=\deg\mu_{i+1}<D_{r,\bar\mu}.
$$
But this contradicts our assumption on $D_{r,\bar\xi}$, and so we have proven the claim.

\medskip

Hence, either $H_{\bar\xi}^{i+1}(d)< H_{\bar\xi}^i(d)$ or $H_{\bar\xi}^i(d)< H_{\bar\mu}^i(d)$. Induction yields then that $H_{\bar\xi}^{i+1}(d)< H_{\bar\mu}^i(d)$, which implies that 
$$
H_{\bar\xi}^{i+1}(d)\leq H_{\bar\mu}^i(d)-1=H_{\bar\mu}^{i+1}(d),
$$
as desired.  

Now let $d\geq\deg \mu_{i+1}$. By Macaulay's theorem
\begin{equation}\label{use1}
H_{\bar\xi}^{i+1}(d+1)\leq H_{\bar\xi}^{i+1}(d)^{\langle d\rangle},
\end{equation}
and
\begin{equation}\label{use2}
H_{\bar\mu}^{i+1}(d+1)=H_{\bar\mu}^{i+1}(d)^{\langle d\rangle}.
\end{equation}
It then follows, by induction on $d\geq \deg\mu_{i+1}$ and property \eqref{sim}, that 
\begin{equation}\label{use3}
H_{\bar\xi}^{i+1}(d)^{\langle d\rangle}\leq H_{\bar\mu}^{i+1}(d)^{\langle d\rangle}.
\end{equation}
Thus, putting \eqref{use1}, \eqref{use2}, and \eqref{use3} together, we get
$$
H_{\bar\xi}^{i+1}(d+1)\leq H_{\bar\mu}^{i+1}(d+1),
$$
and the result follows.

To prove the last statement, note that setting $i=L$ (recall $L = \L^1_{g,m}$) and $d=\deg\mu_L$ yields
$$
H_{\bar\xi}^L(\deg \mu_L)\leq H_{\bar\mu}^L(\deg\mu_L)=0.
$$
Thus, for every $\eta\in\N^m$ with $\deg \eta=\deg\mu_L$ we have that $\eta\geq \xi_j$ for some $\xi_j \in \bar\xi$. This implies that $D_{r,\bar\xi}\leq \deg\mu_L +1=g(\L_{g,m}^1)+1$.
\end{proof}

We can now conclude:

\begin{theorem}\label{UpperBound}
For all $r \geq 0$ we have 
$$
C_{r,m}^1=\L_{g,m}^1+r-1.
$$ 
In particular, if $r\geq 1$ then
\begin{equation}\label{reAc}
C_{r,m}^1=A(m-1,C_{r-1,m}^1)
\end{equation}
and
$$
C_{r,m}^1\leq A(m,r-1)-1,
$$
and if $r\geq 2$ then
$$
A(m,r-2)\leq C_{r,m}^1
$$
where $A$ denotes the Ackermann function.
\end{theorem}
\begin{proof}
By the discussion above, all that is left to prove is the ``in particular'' clause. In Proposition 1.1 of \citep{Socias} Moreno Soc\'ias shows that if $f:\N_{>0}\to\N$ is a function of the form $f(i)=s+i-1$, for some integer $s\geq 1$, then $\L_{f,m}^1=A(m,s-1)-s$. Now, by Proposition \ref{basic}, $C_{r,1}^1=r$; on the other hand, $A(0,C_{r-1,1}^1)=C_{r-1,m}^1+1=r$, so (\ref{reAc}) holds when $m=1$. Assume $m>1$.  Observe that the antichain sequence $\bar\mu$ defined above has the form
\begin{multline*}
(r,0,\dots,0),(r-1,1,0,\dots,0),(r-1,0,2,0,\dots,0),\dots,(1,0,\dots,0,C_{r-1,m}^1-1), \\
(0,C_{r-1,m}^1+1,0,\dots,0), (0,C_{r-1,m}^1,2,0,\dots,0),\dots, (0,0,C_{r,m}^1-1).
\end{multline*}
By the result of Moreno Soc\'ias, the length of the sequence in the second line equals $A(m-1,C_{r-1,m}^1)-C_{r-1,m}^1-1$. Hence, the degree of the last tuple of the sequence equals $A(m-1,C_{r-1,m}^1)-1$. Consequently, $C_{r,m}^1=A(m-1,C_{r-1,m}^1)$, as desired.

Now consider the function $h:\N_{>0}\to\N$ given by $h(i)=r+i-1$. Then $g(i)\leq h(i)$ for all $i$, and so $\L_{g,m}^1\leq \L_{h,m}^1=A(m,r-1)-r$. Hence, 
$$
C_{r,m}^1=\L^1_{g,m}+r-1\leq A(m,r-1)-1.
$$
For the second inequality consider the function $H(i)=r+i-2$. Then $H(i)\leq g(i)$ for all $i$, and so $\L_{g,m}^1\geq \L_{H,m}^1=A(m,r-2)-r+1 $. Hence, 
$$
C_{r,m}^1=\L_{g,m}^1+r-1\geq A(m,r-2).
$$
\end{proof}

We finish this subsection with some computations for small values of $m$.

\begin{corollary}\label{specificvalues}
For any integer $r\geq 0$, we have $C_{r,2}^1=2r$ and $C_{r,3}^1=3(2^r-1)$.
\end{corollary}
\begin{proof}
We prove the result by induction on $r$. The base case $r=0$ is clear. For the induction step we consider the cases for $m=2$ and $m=3$ separately.

\medskip
\noindent \underline{$m = 2$.}  In Proposition \ref{basic}(3) we proved that $C^1_{r,2} = 2r$ using the definition of $C^n_{r,m}$, but Theorem \ref{UpperBound} provides a more direct way of computing it. Assume $C_{r-1,2}^1=2(r-1)$. Then, by (\ref{reAc}) of Theorem \ref{UpperBound}, 
$$
C_{r,2}^1=A(1,C_{r-1,m}^1)=A(1,2r-2)=2r.
$$

\medskip
\noindent \underline{$m=3$.} Assume $C_{r-1,3}^1=3(2^{r-1}-1)$. Then, by \eqref{reAc} above,
$$
C_{r,3}^1=A(2,C_{r-1,3}^1)=A(2,3(2^{r-1}-1))=2(3(2^{r-1}-1))+3=3(2^r-1).
$$

\end{proof}

\subsection{The case $n>1$.}\label{todefgn}

We now extend the results of the previous subsection to arbitrary $n\geq 1$. Let $r_1:=r$ and $g_1:\N_{>0}\to \N$ be defined as $g_1(i)=r$ and $g_1(i)=i+r-2$ for $i\geq 2$. For $n>1$, we define $r_n$ and $g_n:\N_{>0}\to \N$ recursively by 
$$
r_n:=\L_{g_{n-1},m}^{n-1}+r-(n-1)
$$
and 
$$
g_n(i) = \begin{cases} g_{n-1}(i) & \text{ if } i\leq\L_{g_{n-1},m}^{n-1} \\ r_n & \text{ if } i = \L_{g_{n-1},m}^{n-1}+1 \\ i+r_n-\L_{g_{n-1},m}^{n-1}-2 & \text{ if } i\geq \L_{g_{n-1},m}^{n-1}+2 \end{cases}
$$
Note that $r_2=\L_{g_1,m}^1+r-1=C_{r,m}^1$ (by Theorem~\ref{UpperBound}).

We will prove that
\begin{equation}\label{arbn1}
C_{r,m}^n=\L_{g_n,m}^n+r-n.
\end{equation}
This will imply that 
$$
C_{r,m}^n=C_{C_{r,m}^{n-1},m}^1 \quad \text{ for }n\geq 2.
$$
In Proposition \ref{basic} we proved that \eqref{arbn1} holds in the case $r=0$ or $m=1$. We now assume $r\geq 1$ and $m\geq 2$. In this case $\L_{g_n,m}^n\geq \L_{g_{n-1},m}^{n-1}+2$, and so by definition of $r_n$ and $g_n$ we get
$$
g_n(\L_{g_n,m}^n)+1=\L^n_{g_n,m}+r_n-\L_{g_{n-1},m}^{n-1}-1=\L_{g_n,m}^n+r-n.
$$

Thus, to prove \eqref{arbn1} it suffices to prove
\begin{equation}\label{arbn2}
C_{r,m}^n=g_n(\L_{g_n,m}^n)+1.
\end{equation}

Let $\bar\mu = (\mu_1,\dots,\mu_L)$ be the antichain sequence in $\Nn$ defined recursively as follows:
$$
\mu_1 = \max_{\unlhd}\{\alpha \in \Nn : \deg \alpha = g_n(1)\},
$$
and, as long as it is possible, 
$$
\mu_i = \max_{\unlhd}\{\alpha \in \Nn : \deg \alpha = g_n(i) \text{ and } \alpha \not \geq \mu_1,\dots,\mu_{i-1}\}.
$$

In \citep[\S3.3]{LeOv}, it is shown that $\bar\mu$ is an antichain sequence of $\Nn$ having length $L = \L_{g_n,m}^n$ (i.e., $\bar \mu$ is of maximal length among antichain sequences of $\Nn$ with degree growth bounded by $g_n$). It is also observed that 
$$
H_{\bar\mu}^L(\deg \mu_L)=H_{\bar\mu}^L(g_n(L))=0,
$$
where $H_{\bar \mu}^i$ denotes the Hilbert-Samuel function of $\bar \mu$, that is, for $i,d\geq 0$,
$$
H_{\bar\mu}^i(d) = |\{\alpha \in \Nn : \deg \alpha = d \text{ and } \alpha \not \geq \mu_j \text{ for all } j \leq i \text{ for which } \mu_j \text{ is defined}\}|.
$$

The antichain sequence $\bar \mu$ can be more explicitly constructed as follows:

Let $\bar \mu^{(1)}$ be the antichain sequence of maximal length with degree growth bounded by $f_1(i):=g_1(i)$ constructed in Section \ref{subsec1} inside of $\N^m\times\{n\}$ (i.e., the $n$-th copy of $\N^m$ in $\Nn$). Let $L_1$ be the length of $\bar\mu^{(1)}$; in other words $L_1=\L_{f_1,m}^1$. Thus, $\bar\mu^{(1)}$ is of the form 
$$
((\mu^{(1)}_1,n),\dots,(\mu^{(1)}_{L_1},n)).
$$ 
Similarly, let $\bar \mu^{(2)}$ be the antichain sequence of maximal length with degree growth bounded by $f_2(i):=g_2(i+L_1)$ inside of $\N^m\times\{n-1\}$, and let $L_2$ be the length of $\bar\mu^{(2)}$ (that is, $L_2=\L_{f_2,m}^1$). Then, 
$$
\bar\mu^{(2)}=((\mu^{(2)}_1,n-1),\dots,(\mu^{(2)}_{L_2},n-1)).
$$
Continuing in this fashion, we build $\bar \mu^{(j)}$ for $j=3,\dots n$ as the antichain sequence of maximal length with degree bounded growth bounded by 
$$
f_{j}(i)=g_j(i+L_1+\cdots+ L_{j-1})
$$
inside of $\N^m\times\{n-j+1\}$, and let $L_j$ be the length of $\bar\mu^{(j)}$ (that is, $L_j=\L_{f_j,m}^1$). Then the sequence $\bar \mu$ is the concatenation of $\bar \mu^{(1)},\dots,\bar \mu^{(n)}$; in particular, 
$$
\L_{g_n,m}^n=\L_{f_1,m}^1+\cdots+\L_{f_n,m}^1
$$
(see \citep[Proposition 3.13]{LeOv}). Also, note that this implies that 
\begin{equation}\label{tocomp}
\L_{g_n,m}^n=\L_{g_{n-1},m}^{n-1}+\L_{f_n,m}^1.
\end{equation}
From this construction of $\bar\mu$, one obtains the following recursive formula
$$
\L_{g_n,m}^n=\Psi_{f_1,m}(1,(f_1(1),0,\dots,0))+\cdots+\Psi_{f_n,m}(1,(f_n(1),0,\dots,0)),
$$
where $\Psi_{f_j,m}$ is defined as in \eqref{algobo} with $f_j$ in place of $g$.

We now prove \eqref{arbn2}. First, we show that $C_{r,m}^n\geq g(\L_{g_n,m}^n)+1$.

\begin{lemma}
With $\bar\mu$ as above, we have $D_{r,\bar\mu}=g_n(\L_{g_n,m}^n)+1$. In particular, $C_{r,m}^n \ge g_n(\L_{g_n,m}^n)+1$.
\end{lemma}
\begin{proof}
We proceed by induction on $n$. The case $n=1$ is given in Proposition \ref{botu}. Assume it holds for $n-1$. Then $D_{r,\bar\mu'}=g_{n-1}(\L_{g_{n-1},m}^{n-1})+1=r_n$, where $\bar\mu'$ is the concatenation of $\bar\mu^{(1)},\dots,\bar\mu^{(n-1)}$. Since $\bar\mu$ is the concatenation of $\bar\mu'$ and $\bar\mu^{(n)}$, we have that $D_{r,\bar\mu}\geq D_{r,\bar\mu'}$. Thus, by Remark~\ref{changevar}, $D_{r,\bar\mu}=D_{r_n,\bar\mu}$. It follows that $D_{r,\bar\mu}=D_{r_n,\bar\mu^{(n)}}$. Since $\deg\mu^{(n)}_1=r_n$, Proposition \ref{botu} (applied with $\bar\mu^{(n)}$  and $r_n$ in place of $\bar\mu$ and $r$, respectively) yields 
$$
D_{r_n,\bar\mu^{(n)}}=\deg\mu^{(n)}_{L_n}+1=g_n(\L_{g_n,m}^n)+1,
$$
as desired. 
\end{proof}

We now prove that $C_{r,m}^n\leq g_n(\L_{g_n,m}^n)+1$. To do this, suppose there is an antichain sequence $\bar\alpha=(\alpha_1,\dots,\alpha_M)$ of $\Nn$ such that $D_{r,\bar\alpha}\geq g_n(\L_{g_n,m}^n)+1$. We must show that then $D_{r,\bar\alpha}\leq g_n(\L_{g_n,m}^n)+1$. We prove this in Theorem~\ref{theforn} below which establishes the relationship between the Hilbert-Samuel functions of $\bar\mu$ and $\bar\alpha$. We will make use of the following combinatorial lemma (for a proof see \citep[Lemma~1.3]{coLeOv}).

\begin{lemma}\label{techlem1}
Let $m$ and $d$ be positive integers. Suppose $a_1,\dots, a_t$ and $b_1,\dots,b_s$ are sequences of nonnegative integers such that 
$$b_1\leq b_2 =\cdots=b_s=\binom{m-1+d}{d}$$ 
and $b_s\geq a_i$ for all $i\leq t$. If $a_1+\cdots+a_t\leq b_1+\cdots+b_s$, then
$$
a_1^{\langle d\rangle}+\cdots+a_t^{\langle d\rangle}\leq b_1^{\langle d\rangle}+\cdots+b_s^{\langle d\rangle}.
$$
\end{lemma}

\begin{theorem}\label{theforn}
With $\bar\mu$ and $\bar\alpha$ as above, we have that 
$$
H_{\bar\alpha}^i(d)\leq H_{\bar\mu}^i(d)
$$
for all $i,d\geq 0$. As a result, $D_{r,\bar\alpha}\leq g_n(\L_{g_n,m}^n)+1$.
\end{theorem}
\begin{proof}
First we make some observations. For any antichain sequence $\bar\beta$ of $\Nn$ and each $1\leq j\leq n$, we let $H_{\bar\beta}^{i,j}$ be the Hilbert-Samuel function of the subsequence of $\bar\beta$ consisting of its elements inside of $\N\times\{n-j+1\}$ (i.e., the $(n-j+1)$-th copy of $\N^m$ in $\Nn$). Then
\begin{equation}\label{usep0}
H_{\bar\beta}^{i}(d)=H_{\bar\beta}^{i,1}(d)+\cdots+H_{\bar\beta}^{i,n}(d).
\end{equation}
By the construction of $\bar\mu$, we have that 
$$
H_{\bar\mu}^{i,j}(d)=H_{\bar\mu^{(j)}}^{i}(d).
$$
Thus, if $\L_{g_j,m}^{j}< i\leq\L_{g_{j+1},m}^{j+1}$, for some $0\leq j<n$, then for $d\geq \deg\mu^{(j)}_{L_j}$ we have
\begin{equation}\label{ineq23}
0=H_{\bar\mu}^{i,0}(d)=\cdots=H_{\bar\mu}^{i,j}(d)\leq H_{\bar\mu}^{i,j+1}(d)\leq H_{\bar\mu}^{i,j+2}(d)=\cdots=H_{\bar\mu}^{i,n}(d)=\binom{m-1+d}{d}
\end{equation}
where recall that the displayed binomial equals the number of $m$-tuples of degree $d$. For the case $j=0$, we are setting $\L_{g_0,m}^{0}=0$, $\mu^{(0)}_{L_0}=(0,\dots,0)$, and $H_{\bar\mu}^{i,0}(d)=0$.  We note that the inequalities in \eqref{ineq23} will allow us to apply Lemma~\ref{techlem1} below with the $H_{\bar \alpha}^{i,j}(d)$'s in place of the $a$'s and the $H_{\bar\mu}^{i,j}(d)$'s in place of the $b$'s.

We now go back to the proof of the theorem. We proceed by induction on $i$. For the base case $i=0$, we have
$$
H_{\bar\alpha}^0(d)=n\cdot\binom{m-1+d}{d}=H_{\bar\mu}^0(d).
$$
Now assume the inequality holds for $i\geq 0$. We prove it for $i+1$. Note that, since $D_{r,\bar\alpha}\geq g_n(\L_{g_n,m}^n)+1$, the sequence $\bar\alpha$ contains at least two elements of degree at most $r$. It follows then that $H_{\bar\alpha}^1(d)\leq H_{\bar\mu}^1(d)$ and $H_{\bar\alpha}^2(d)\leq H_{\bar\mu}^2(d)$ for all $d\geq 0$. Thus, we assume $i\geq 2$. 

We have that for $d<\deg\mu_{i+1}$
$$
H_{\bar\alpha}^{i+1}(d)\leq H_{\bar\alpha}^{i}(d)\leq H_{\bar\mu}^{i}(d)=H_{\bar\mu}^{i+1}(d).
$$
Now consider the case $d=\deg\mu_{i+1}$ (note that $d>1$ since $r>0$ and $i\geq 2$). Let $0\leq j< n$ be such that $\L_{g_j,m}^{j}< i+1\leq \L_{g_{j+1},m}^{j+1}$. Note that then $d\geq \deg\mu^{(j)}_{L_j}$.

\medskip
\noindent {\bf Claim.} Either $H_{\bar\alpha}^{i+1}(d)< H_{\bar\alpha}^i(d)$ or $H_{\bar\alpha}^i(d)<H_{\bar\mu}^i(d)$.

\medskip
\noindent {\it Proof of Claim.} Towards a contradiction suppose
\begin{equation}\label{equn}
H_{\bar\alpha}^{i+1}(d)=H_{\bar\alpha}^i(d)=H_{\bar\mu}^i(d).
\end{equation}
By induction on $d$, we have $\displaystyle \sum_{k=1}^nH_{\bar\alpha}^{i,k}(d-1)\leq\sum_{k=1}^n H_{\bar\mu}^{i,k}(d-1)$. So Lemma \ref{techlem1} (which can be applied by \eqref{ineq23}) yields $\displaystyle \sum_{k=1}^nH_{\bar\alpha}^{i,k}(d-1)^{\l d-1\r}\leq\sum_{k=1}^n H_{\bar\mu}^{i,k}(d-1)^{\l d-1\r}$. Using the fact that the $\bar\mu^{(j)}$'s are all compressed, Macaulay's theorem (Theorem~\ref{Mac}) gives $\displaystyle \sum_{k=1}^n H_{\bar\mu}^{i,k}(d-1)^{\l d-1\r}=\sum_{k=1}^n H_{\bar\mu}^{i,k}(d)=H_{\bar\mu}^i(d)$. Putting these (in)equalities together, we get
$$
\sum_{k=1}^nH_{\bar\alpha}^{i,k}(d-1)^{\l d-1\r}\leq\sum_{k=1}^n H_{\bar\mu}^{i,k}(d-1)^{\l d-1\r}=H_{\bar\mu}^i(d)=H_{\bar\alpha}^i(d).
$$
This inequality and Macaulay's theorem imply that 
$$
H_{\bar\alpha}^{i,k}(d)=H_{\bar\alpha}^{i,k}(d-1)^{\l d-1\r}
$$
for $k=1,\dots,n$. These equalities, together with $H_{\bar\alpha}^{i+1}(d)=H_{\bar\alpha}^i(d)$, imply that $\deg \alpha_s\neq d$ for all $s\leq i+1$ for which $\alpha_s$ is defined. This fact and Corollary \ref{HSdeg} imply that 
$$
D_{r,\bar\alpha}\leq d=\deg\mu_{i+1}<D_{r,\bar\mu}.
$$
However, this contradicts our assumption on $D_{r,\bar\alpha}$, and so we have proven the claim.

\medskip

Hence, either 
$$
H_{\bar\alpha}^{i+1}(d)< H_{\bar\alpha}^i(d)\; \text{ or } \; H_{\bar\alpha}^i(d)< H_{\bar\mu}^i(d).
$$
Induction yields then that $H_{\bar\alpha}^{i+1}(d)< H_{\bar\mu}^i(d)$, which implies that 
$$
H_{\bar\alpha}^{i+1}(d)\leq H_{\bar\mu}^i(d)-1=H_{\bar\mu}^{i+1}(d),
$$
as desired.  

Now let $d\geq\deg \mu_{i+1}$ (note that then $d\geq \deg\mu^{(j)}_{L_j}$). By Macaulay's theorem, 
\begin{equation}\label{usep1}
H_{\bar \alpha}^{i+1,k}(d+1)\leq H_{\bar \alpha}^{i+1,k}(d)^{\langle d\rangle}
\end{equation}
and
\begin{equation}\label{usep2}
H_{\bar \mu}^{i+1,k}(d+1)=H_{\bar \mu}^{i+1,k}(d)^{\langle d\rangle}.
\end{equation}
for $k=1,\dots,n$. It then follows, by induction on $d\geq \deg\mu_{i+1}$ and using Lemma~\ref{techlem1}, that 
\begin{equation}\label{usep3}
H_{\bar \alpha}^{i+1,1}(d)^{\langle d\rangle}+\cdots+H_{\bar \alpha}^{i+1,n}(d)^{\langle d\rangle}\leq H_{\bar \mu}^{i+1,1}(d)^{\langle d\rangle}+\cdots +H_{\bar \mu}^{i+1,n}(d)^{\langle d\rangle}.
\end{equation}
Thus, putting \eqref{usep0}, \eqref{usep1}, \eqref{usep2}, and \eqref{usep3} together, we get
$$
H_{\bar \alpha}^{i+1}(d+1)\leq H_{\bar \mu}^{i+1}(d+1),
$$
and the result follows.

To prove the last statement, note that setting $i=L$ (recall $L = \L^n_{g_n,m}$) and $d=\deg\mu_L$ yields
$$
H_{\bar\alpha}^L(\deg \mu_L)\leq H_{\bar\mu}^L(\deg\mu_L)=0.
$$
Thus, for every $\beta\in\Nn$ with $\deg \beta=\deg\mu_L$ we have that $\beta\geq \alpha_j$ for some $\alpha_j \in \bar\alpha$. This implies that $D_{r,\bar\alpha}\leq \deg\mu_L +1=g_n(\L_{g_n,m}^n)+1$.
\end{proof}

We can now conclude:
\begin{theorem}\label{UpperBound2}
For all $r\geq 0$, we have
$$
C_{r,m}^n=\L_{g_n,m}^n+r-n.
$$
As a result,
$$
C_{r,m}^n=C_{C_{r,m}^{n-1},m}^1 \quad \text{ for } n\geq 2,
$$
and so, in particular,
$$
C_{r,2}^n=2^n r.
$$
\end{theorem}
\begin{proof}
By the discussion above, all that is left to show is the ``as a result'' clause. Note that $f_n(1)=C_{r,m}^{n-1}$ and $f(i)=i+C_{r,m}^{n-1}-2$ for $i\geq 2$. Thus, by Theorem \ref{UpperBound}, $C_{C_{r,m}^{n-1},m}^1=\L_{f_n,m}^1+C_{r,m}^{n-1}-1$. By \eqref{tocomp}, we thus have
$$
C_{r,m}^n=\L_{g_n,m}^n+r-n=\L_{f_n,m}^1+\L_{g_{n-1},m}^{n-1}+r-n=\L_{f_n,m}^1+C_{r,m}^{n-1}-1=C_{C_{r,m}^{n-1},m}^1.
$$
For the computation of $C_{r,2}^n$ we proceed by induction on $n$. The base case $n=1$ is in Corollary~\ref{specificvalues}. For the induction step we have
$$
C_{r,2}^n=C_{C_{2,m}^{n-1},m}^1=C_{2^{n-1}r,m}^1=2(2^{n-1}r)=2^nr.
$$
\end{proof}

Define the function $A_n:\N \times \N_{>0} \to \N$ by
$$
A_n(x,y) = \begin{cases} A(x,y-1) - 1 & \text{ if } n = 1 \\ A(x,A_{n-1}(x,y)-1)-1 & \text{ if } n > 1 \end{cases}
$$
where $A$ denotes the Ackermann function.  We then have the following:
\begin{corollary}\label{generalAckermannbound}
For all $r \ge 1$, we have
$$
C^n_{r,m} \le A_n(m,r).
$$
Additionally, if $r \ge 2$, then
$$
A_n(m,r-1) + 1 \le C^n_{r,m}.
$$
\end{corollary}
\begin{proof}
We prove both inequalities by induction on $n$.  The base case $n=1$ is given by Theorem \ref{UpperBound}.  Now suppose both inequalities are true for $n-1$.  Then, by induction and Theorems \ref{UpperBound} and \ref{UpperBound2}, we get
$$
C^n_{r,m} = C^1_{C^{n-1}_{r,m},m} \le A(m,C^{n-1}_{r,m} - 1) -1 \le A(m,A_{n-1}(m,r)-1) - 1 = A_n(m,r)
$$
and, if $r \ge 2$,
$$
C^n_{r,m} = C^1_{C^{n-1}_{r,m},m} \ge A(m,C^{n-1}_{r,m} - 2) \ge A(m,A_{n-1}(m,r-1)-1) = A_n(m,r-1) + 1.
$$
\end{proof}

\section{Some applications}\label{applications}

In this section we present several applications using the upper bound $C_{r,m}^n$ of $T_{r,m}^n$. We assume some familiarity with the differential ring of differential polynomials and with the notion of a characteristic set of a differential ideal. For the unfamiliar reader we suggest Chapters I and IV of  \citep{Kolchin}. 

Throughout this section we fix an $n$-tuple of differential indeterminates $x=(x_1,\dots,x_n)$.  For any set of differential polynomials $f_1,\dots,f_s\in K\{x\}$ over the differential field $(K,\partial_1,\dots,\partial_m)$, we let 
$$
(f_1,\dots,f_s) \; \text{ and } \; [f_1,\dots,f_s]
$$
denote the ideal and the differential ideal generated by the $f_i$'s, respectively.

In our first application we prove that if $S$ is a collection of differential polynomials in $n$ indeterminates of order at most $r$, then each minimal prime differential ideal containing $S$ has a characteristic set whose elements have order at most $C_{r,m}^n$. Next, we show that if $V \subseteq \mathbb{A}^n$ is a differential algebraic variety defined by differential polynomials of order at most $r$, then every finite order irreducible component of $V$ has order at most $n\cdot (C^n_{r,m})^m$.

The final two applications deal with classical problems in effective differential algebra, namely, the effective differential Nullstellensatz \citep{GKO} and the effective computation of B\'ezout-type estimates for differential algebraic varieties \citep{FreLe}. The computation of the appropriate bound in each of these problems depends on $T_{r,m}^n$ as it relies on an algebro-geometric characterization of differentially closed fields (see \citep[Proposition 4.1]{FreLe}), which is essentially a geometric translation of the definition of $T_{r,m}^n$.

\subsection{Order bounds for characteristic sets} \

Here the notion of characteristic set will be with respect to the canonical orderly ranking on the set of algebraic indeterminates $\{\partial^\xi x_i: (\xi,i)\in \Nn\}$. Characteristic sets of prime differential ideals can be computed by means of several differential-algebraic algorithms (for instance, from various modifications of the Rosenfeld-Gr\"obner algorithm). It is thus important to compute good estimates for the order of elements of a characteristic set in terms of the order of the original differential polynomials. The first attempt at this problem was made in 1956 by Seidenberg \citep[\S14]{Seidenberg} (where a solution was only suggested). Recently an explicit upper bound was found; in \citep[Proposition 1]{Kondratieva} it is stated that an upper bound for the order of the elements of a characteristic set of a prime differential ideal $P$ is
$$
A(m+7, \max(9,n,2^{9r},d)-1)
$$
where $A$ is the Ackermann function, and $r$ and $d$ bound the order and degree, respectively, of a set of radical differential generators of $P$. Our results yield an improvement of the above bound. Moreover, the bound presented here does \emph{not} depend on the degrees of the given collection of differential polynomials. It is worth pointing out that even the existence of such a bound, with no assumption on the degrees, seems to be a new result. 

Recall that by Corollary \ref{generalAckermannbound},
$$
C_{r,m}^n\leq A_n(m,r)
$$
where $A_1(x,y)=A(x,y-1)-1$ and $A_n(x,y)=A(x,A_{n-1}(x,y)-1)-1$ for $n>1$.

\begin{proposition}\label{ordchar}
Suppose $S$ is a collection of differential polynomials in $K\{x\}$ of order at most $r$. Then, each minimal prime differential ideal $P$ containing $S$
has a characteristic set whose elements have order at most $C_{r,m}^n$.
\end{proposition}
\begin{proof}
Let $(\mathbb U, \partial_1,\dots,\partial_m)$ be a \emph{universal} differential field extension of $K$; that is, every irreducible differential algebraic variety over $K$ has a differential generic point in $\mathbb U$. Let 
$$W=\{u\in \mathbb U^n: f(u)=0 \text{ for all } f\in S\}$$ 
and  
$$V=\{u\in\mathbb U^n: f(u)=0 \text{ for all } f\in P\}.$$ 
Then $V$ is an irreducible differential algebraic variety over $K$ which is a component of $W$. Let $a=(a_1,\dots,a_n)$ be a differential generic point of $V$ over $K$, and set $a_i^{\xi}=\partial^\xi a_i$ for all $(\xi,i)\in \Nn$. Then the differential kernel $K(a_i^\xi:(\xi,i)\in \Gamma(r))$ has a prolongation of length $C_{r,m}^n$, namely $K(a_i^\xi:(\xi,i)\in \Gamma(C_{r,m}^n))$.

By Theorem~\ref{construct}, there is $r\leq h\leq C_{r,m}^n $ such that the differential kernel $L:=K(a_i^\xi: (\xi,i)\in \Gamma(h))$ has a principal realization, call it $b=(b_1,\dots,b_n)$. By universality of $\mathbb U$ over $K$ we can assume that $b$ is a tuple from $\mathbb U$. Since $(\partial^\xi b_i:(\xi,i)\in \Gamma(h))$ is a generic specialization of $(\partial^\xi a_i:(\xi,i)\in \Gamma(h))$ (in the algebraic sense) and the order of the elements of $S$ is bounded by $r\leq h$, the tuple $b$ is in $W$. On the other hand, Lemma \ref{diffspec} shows that the tuple $a$ is a differential specialization of $b$. Thus, the tuple $a$ is in the irreducible differential variety $Z\subseteq W$ which has $b$ as a differential generic point over $K$. But, since $V$ is a component of $W$, this implies that $Z=V$. We have thus shown that $b$ is a differential generic point of $V$. Thus, if a characteristic set of $P$ had an element of order larger than $C_{r,m}^n\geq h$ then the differential field $K\l b\r$ would 
have a minimal leader of degree larger than $h$, contradicting the fact that $b$ is a principal realization of $L$ (recall that this means that all the minimal leaders of $K\l b\r$ are contained in $L$).
\end{proof}

As we have seen in previous sections, for small values of $m$ this upper bound is computationally feasible. For instance, this establishes that for the case of two derivations ($m=2$) such characteristic sets have order at most $2^n r$ (before there was no accessible bound for $m\geq 2$).

\begin{remark}
In the case of linear differential equations (i.e., when all the differential polynomials in $S$ have degree one), one can produce a double-exponential bound for the order of the elements of a characteristic set of each prime component of $S$. This can be achieved using results of Chistov and Grigoriev. Indeed, as the differential polynomials are linear, $S$ generates a prime differential ideal $P$, any characteristic set of which is linear; then the K\"ahler differential $d$ maps a characteristic set of $P$ to a Gr\"obner basis of the D-module $d(P)$ and the orders are preserved. At this point one can apply the double-exponential bounds for Gr\"obner bases obtained in \citep{CG} to obtain similar bounds for the order of the elements of the characteristic set\footnote{We thank an anonymous referee for pointing out this strategy for the linear case.}. We leave the details of this argument to the interested reader.
\end{remark}

\newpage
\subsection{On the order of a differential algebraic variety} \

Given an (affine) differential algebraic variety $V\subseteq \mathbb A^n$ defined over $K$, we let $\mathcal I(V)$ denote its defining differential ideal; that is, 
$$
\mathcal I(V):=\{f\in K\{x\}: f(V)=0\}.
$$
Recall that the differential coordinate ring of $V$ is defined as
$$
K\{V\}:=K\{x\}/\mathcal I(V),
$$
and, if $V$ is irreducible, then the differential function field is $K\langle V \rangle:=\operatorname{Frac}(K\{V\})$. The \emph{order} of $V$ is
$$
\ord(V) = \operatorname{trdeg}_K(K\langle V\rangle).
$$
Note that in the ordinary case, $m=1$, an irreducible differential algebraic variety $V$ has finite order if and only it is of (differential) dimension zero.

Now, let $V$ be an arbitrary (not necessarily irreducible) differential algebraic variety over $K$, and $W$ be a $K$-irreducible component of $V$. The following question arises: {\it If $W$ is of finite order, can we find an upper bound for $\ord(W)$ in terms of data explicitly obtained from $V$?}

The purpose of this section is to present a positive answer to the above question. More precisely,

\begin{proposition}\label{mainprop}
Suppose $V\subseteq \mathbb A^n$ is a differential algebraic variety defined by differential polynomials over $K$ of order at most $r$. If $W$ is a $K$-irreducible component of $V$ of finite order, then 
$$
\ord(W) \leq n\cdot (C_{r,m}^n)^m.
$$
\end{proposition}

\begin{remark}
In the ordinary case, $m=1$, the upper bound of Proposition \ref{mainprop} reduces to $nr$, which appears in Ritt's book \citep[Chapter 6]{Ritt}. However, up until now, no upper bound had been computed for $m>1$.
\end{remark}

We will make use of the following fact due to Kolchin. We work, as in the previous section, with the canonical orderly ranking on the set of algebraic indeterminates, and we assume familiarity with the notion of the Kolchin polynomial of a prime differential ideal. Recall that if $E$ is any subset of $\mathbb N^m$ and $U_E$ is the set of elements of $\mathbb N^m$ that are not greater than or equal to any element of $E$ (with respect to the product order of $\mathbb N^m$), then there exists a numerical polynomial $\omega_E$ such that for sufficiently large $t$ the number of elements $(u_1,\dots,u_m)\in U_E$ with $\sum u_i\leq t$ is equal to $\omega_E(t)$. 

\begin{fact}[\citep{Kolchin},\S II.12]\label{usen1}
If $\mathcal P$ is a prime differential ideal of $K\{x_1\dots,x_n\}$, $\Lambda$ is a characteristic set of $\mathcal P$, and we let $E_i$ denote the set of all $\xi\in\mathbb N^m$ such that $\partial^\xi x_i$ is a leader of $\Lambda$, then
$$
\omega_{\mathcal P}=\sum_{i=1}^n \omega_{E_i}
$$
where $\omega_{\mathcal P}$ denotes the Kolchin polynomial of $\mathcal P$.
\end{fact}

\begin{proof}[Proof of Proposition \ref{mainprop}]
Note that, since $W$ is of finite order, $\omega_{\mathcal P}$ is a constant numerical polynomial equal to $\ord(W)$. Let $\mathcal S$ be a collection of differential polynomials over $K$ defining $V$ whose elements have order at most $r$. We also let $\mathcal P:=\mathcal I(W)$, $\Lambda$ be a characteristic set for $\mathcal P$, and $E_i$ be as in Fact \ref{usen1}.  By Fact \ref{usen1}, it suffices to show that 
$$
\omega_{E_i}\leq (C_{r,m}^n)^m
$$ 
for all $i=1,\dots,n$.  

Note that, again by Fact \ref{usen1}, each $\omega_{E_i}$ is constant. This implies that $U_{E_i}$ is finite (recall that $U_{E_i}$ is the set of elements of $\mathbb N^m$ that are not greater than or equal to any element of $E_i$). Therefore, $\omega_{E_i}=|U_{E_i}|$. Let $\xi=(u_1,\dots,u_m)\in U_{E_i}$; we claim that $u_j < C_{r,m}^n$ for all $j$. Recall that for each $j=1,\dots,m$, we let $\bf j$ denote the $m$-tuple with a $1$ in the $j$-th entry and zeros elsewhere, and note that, since $U_{E_i}$ is finite, there is a sufficiently large $t$ such that $t{\bf j}$ is greater than or equal to some element $\eta$ of $E_i$. By Proposition~\ref{ordchar}, the elements of a characteristic set of $\mathcal P$ have order at most $C_{r,m}^n$. So, by definition of $E_i$, the $m$-tuple $\eta$ must be of the form $s_j{\bf j}$ with $s_j\leq C_{r,m}^n$. Hence, if $u_j\geq C_{r,m}^n$, we would get that $\xi$ is greater than or equal to $s_j{\bf j}=\eta$, which is impossible. We have thus shown that 
$$
U_{E_i}\subseteq \Omega (C_{r,m}^n)
$$
where $\Omega(s):=\{(v_1,\dots,v_m)\in \mathbb N^m: v_j<s \text{ for all }j\}$ for $s\in\mathbb N$. The latter set clearly has cardinality $s^m$. This yields
$$
|U_{E_i}|\leq |\Omega(C_{r,m}^n)|=(C_{r,m}^n)^m.
$$
as desired.
\end{proof}

\subsection{Effective differential Nullstellensatz} \

Given a system of differential polynomial equations $f_1 = 0, \dots, f_s = 0$, the differential Nullstellensatz states that this system is consistent if and only if 1 cannot be written as a combination of the $f_i$'s and their derivatives.  Making this precise, given a differential field $(K,\partial_1,\dots,\partial_m)$ and a collection of differential polynomials $F = \{f_1,\dots,f_s\} \subseteq K\{x\}$ with $x=(x_1,\dots,x_n)$, the (weak) differential Nullstellensatz states that the system $F = 0$ has a solution in some differential field extension of $K$ if and only if $1 \notin [F]$.  

For this result to be practical, one needs to determine how many times the $f_i$'s must be differentiated.  To this end, for any integer $b\geq 0$, we define 
$$
F^{(b)} := \{\partial^{\xi} f : f \in F \text{ and } \deg\xi\leq b\},
$$
the collection of all derivatives of elements of $F$ up to order $b$.  Then, the \emph{effective differential Nullstellensatz} is concerned with the problem of finding the smallest nonnegative integer $B(m,n,r,d)$ such that for any system of differential polynomials $F \subseteq K\{x\}$, of total order $r$ and degree $d$, we have 
$$
1 \in [F] \; \text{ if and only if } \;1 \in \left(F^{(B(m,n,h,d))}\right).
$$

It has been shown that $B(m,n,r,d)$ is at least exponential in $n$ and double-exponential in $m$ (see \citep{GKO, Sadik}).  The most recent upper bound, given in \citep{GKO}, depends on 
$$
T := \begin{cases} r+1 & \text{ if } m = 1 \\ T^n_{r,m} & \text{ if } m > 1.\end{cases}
$$
The reason one needs $T=r+1$ when $m=1$ instead of $T^n_{r,1} = r$ is due to a technical aspect of the proof of the upper bound.  The upper bound is given by
$$
B(m,n,r,d) \le (n \alpha_{T-1}d)^{2^{O\left(n^3\alpha_T^3\right)}},
$$
where, for any $\ell$, $\alpha_\ell = \binom{\ell+m}{m}$.  While $T$ grows in terms of the Ackermann function, for small numbers of derivations $m$ and differential indeterminates $n$ this bound produces practical results.  For example, when $m=1$, then $T = r+1$, so in this case
$$
B(1,n,r,d) \le (n(r+1)d)^{2^{O\left(n^3(r+2)^3\right)}}.
$$
When $m=2$, then $T = T^n_{r,m} \le 2^nr$, so in this case
$$
B(2,n,r,d) \le \left(n2^{n-1}r\left(2^nr+1\right)d\right)^{2^{O\left(n^32^{6n}r^6\right)}}.
$$

\subsection{B\'ezout-type estimates} \

In differential algebraic geometry, B\'ezout-type estimates refer to the process of computing an upper bound for the degree of the Zariski closure of an affine differential algebraic variety in terms of the orders and degrees of the defining differential polynomials. In \citep{FreLe}, such a bound was computed, and it is given in terms of $T^n_{r,m}$. For the rest of this section we assume that our differential field $(K,\partial_1,\dots,\partial_m)$ is differentially closed. Recall that $x = (x_1,\dots,x_n)$ is an $n$-tuple of differential indeterminates.

For each integer $r \geq 0$, define $\nabla_r:K^n \to K^{n \cdot \alpha_r}$ by
$$
\nabla_r(x) = (\partial^\xi x_i: (\xi,i)\in \Gamma(r)),
$$
where $\alpha_r:=\binom{r+m}{m}$.  Any affine differential variety $Z\subseteq K^n$ defined by differential polynomials of order at most $r$ can be written in the form
$Z = \{v \in V : \nabla_r(v) \in W\}$
for some affine algebraic varieties $V \subseteq K^n$ and $W \subseteq K^{n \cdot \alpha_r}$. Let $\bar Z$ be the Zariski closure of $Z$ in $K^n$. Recall that, for an irreducible variety $U \subseteq K^n$, the {\it degree} of $U$ is given by
\begin{align*}
\deg U = \max \{|U \cap H_1 \cap \cdots \cap H_d| : &\text{ each } H_i \subseteq K^n \text{ is a hyperplane} \\
& \text{ and the intersection is finite }\},
\end{align*}
and if $U$ is not irreducible, then its degree is the sum of the degrees of its irreducible components.  In \citep[\S4]{FreLe} an upper bound for $\deg(\bar Z)$ was established, having the form
$$
\deg(\bar Z) \leq (\deg V)^{\alpha_{r-1}\cdot \alpha_{T'}(m+1)^{d'\alpha_{T'-1}-1}}(\deg W)^{\alpha_{T'-1}\cdot \frac{(m+1)^{d'\alpha_{T'-1}}-1}{m}},
$$
where $d' = \alpha_{r-1}\cdot \dim V$ and $T' = T^{n \cdot \alpha_{r-1}}_{1,m}$.  

Due to our results, one can compute some explicit values for small values of $m$. For example, when $m=1$, then $T' = T^{nr}_{1,1} = 1$, so we get
$$
\deg(\bar Z) \le (\deg V)^{r2^{r\cdot\dim V}}\left(\deg W\right)^{2^{r\cdot \dim V}-1}.
$$
When $m=2$, then $T' = T^{n \cdot \alpha_{r-1}}_{1,2} \le 2^{\frac{nr(r+1)}{2}}$, so the upper bound for $\deg(\bar Z)$ is triple-exponential in $n$ and $r$.

\bibliographystyle{elsarticle-harv}

\end{document}